\documentclass{amsart}
\usepackage{amssymb}
\usepackage{amsmath}
\usepackage{thmtools}
\usepackage{thm-restate}
\usepackage{epsfig}
\usepackage{tikz}
\usepackage{multirow}

\newtheorem{theorem}{Theorem}[section]
\newtheorem{prop}[theorem]{Proposition}
\newtheorem{lemma}[theorem]{Lemma}

\theoremstyle{definition}
\newtheorem{rem}[theorem]{Remark}
\newtheorem{defi}[theorem]{Definition}

\newcommand{\ra}{\rightarrow}

\newcommand{\IC }{\mathbb{C}}
\newcommand{\IR}{\mathbb{R}}

\newcommand{\IZ}{\mathbb{Z}}
\newcommand{\IQ}{\mathbb{Q}}

\newcommand{\coloneqq}{:=}
\newcommand{\rk}{\mathrm{rk}}
\newcommand{\sign}{\mathrm{sign}}
\newcommand{\spinn}{\mathrm{sn}_{\mathbb{R}}}
\newcommand{\hsk}{K3^{\left[2\right]}}
\newcommand{\hskt}{K3^{\left[3\right]}}
\newcommand{\hskq}{K3^{\left[4\right]}}
\newcommand{\hskn}{K3^{\left[n\right]}}

\DeclareMathOperator{\id}{id}
\DeclareMathOperator{\discr}{discr}

\DeclareMathOperator{\Hom}{Hom}

\DeclareMathOperator{\Mo}{Mon}
\DeclareMathOperator{\aut}{Aut}
\DeclareMathOperator{\signt}{sign}
\DeclareMathOperator{\stabil}{Stab}
\DeclareMathOperator{\ns}{NS}
\DeclareMathOperator{\hilb}{Hilb}
\DeclareMathOperator{\grass}{Grass}
\DeclareMathOperator{\pic}{Pic}
\DeclareMathOperator{\alb}{Alb}
\DeclareMathOperator{\br}{Br}
\DeclareMathOperator{\trans}{Tr}

\makeatletter
\def\blfootnote{\xdef\@thefnmark{}\@footnotetext}

\begin{document}

\title[Non-symplectic automorphisms of $\hskn$-type manifolds]{Non-symplectic automorphisms of odd prime order on manifolds of $\hskn$-type}
\author{Chiara Camere and Alberto Cattaneo}

\address{Chiara Camere, Universit\`a degli Studi di Milano,
Dipartimento di Matematica ``F. Enriques'',
Via Cesare Saldini 50, 20133 Milano, Italy} 
\email{chiara.camere@unimi.it}
\urladdr{http://www.mat.unimi.it/users/camere/en/index.html}
\address{Alberto Cattaneo, Universit\`a degli Studi di Milano,
Dipartimento di Matematica ``F. Enriques'',
Via Cesare Saldini 50, 20133 Milano, Italy; \newline Laboratoire de Math\'ematiques et Applications, UMR CNRS 7348, Universit\'e de Poitiers, T\'el\'eport 2, Boulevard Marie et Pierre Curie, 86962 Futuroscope Chasseneuil, France}
\email{alberto.cattaneo1@unimi.it; alberto.cattaneo@math.univ-poitiers.fr}

\maketitle


\begin{abstract}
We classify non-symplectic automorphisms of odd prime order on irreducible holomorphic symplectic manifolds which are deformations of Hilbert schemes of any number $n$ of points on $K3$ surfaces, extending results already known for $n=2$. In order to do so, we study the properties of the invariant lattice of the automorphism (and its orthogonal complement) inside the second cohomology lattice of the manifold. We also explain how to construct automorphisms with fixed action on cohomology: in the cases $n=3,4$ the examples provided allow to realize all admissible actions in our classification. For $n=4$, we present a construction of non-symplectic automorphisms on the Lehn--Lehn--Sorger--van Straten eightfold, which come from automorphisms of the underlying cubic fourfold.   
\end{abstract}

\section{Introduction}

The study of automorphisms of $K3$ surfaces has been a very active research field for decades. The global Torelli theorem  allows to reconstruct automorphisms of a $K3$ surface $\Sigma$ from Hodge isometries of $H^2(\Sigma, \IZ)$ preserving the intersection product; this link, together with the seminal works of Nikulin \cite{nikulin}, \cite{nikulin2}, provided the instruments to investigate finite groups of automorphisms on $K3$'s. In recent years, the interest in automorphisms has extended from $K3$ surfaces to manifolds which generalize them in higher dimension, namely irreducible holomorphic symplectic (IHS) varieties. Results by Huybrechts, Markman and Verbitsky, which provide an analogous of the Torelli theorem for these manifolds, allow to use similar methods, studying the action of an automorphism on the second cohomology group with integer coefficients (which carries again a lattice structure, provided by the Beauville--Bogomolov--Fujiki quadratic form).

A great number of results are known for automorphisms of prime order on IHS fourfolds that are deformations of the Hilbert scheme of two points on a $K3$ surface (so-called manifolds of $\hsk$-type). The symplectic case (i.e.\ automorphisms which preserve the symplectic form) is covered in \cite{camere_sympl_inv} and \cite{mongardi_thesis}; in turn, the study of non-symplectic automorphisms was started by Beauville \cite{beauville_inv} and has seen many relevant contributions, culminating in a complete classification of their action on cohomology (\cite{bcs}, \cite{bcms_p=23}, \cite{tari}). Explicit constructions of automorphisms realizing all possible actions in this classification have been exhibited throughout the years (see \cite{o'grady_epw}, \cite{bcs}, \cite{mw}, \cite{ckkm}), with the exception of the (unique) automorphism of order $23$ whose existence is proved in \cite{bcms_p=23}.

Far less is known, in general, about non-symplectic automorphisms of manifolds of $\hskn$-type when $n \geq 3$. In \cite{joumaah}, Joumaah studies moduli spaces of manifolds $X$ of $\hskn$-type with non-symplectic involutions $\iota: X \ra X$, providing also a classification for the invariant lattice $H^2(X, \IZ)^{\iota^*} = \left\{ v \in H^2(X, \IZ): \iota^*(v)=v \right\}$. This classification, however, is not entirely correct: in the upcoming paper \cite{CCC}, by the two authors and Andrea Cattaneo, we will rectify the mistakes and provide some additional insight on non-symplectic involutions.

In this paper we construct a general theory of non-symplectic automorphisms of odd prime order $p$ of manifolds of $\hskn$-type, for any $n \geq 2$.  

In the case of fourfolds of $\hsk$-type, the authors of \cite{bcs} discovered that the classification of non-symplectic automorphisms was fundamentally richer for $p=2$, rather than for $p$ odd. In the more general setting of manifolds of $\hskn$-type, we show that many additional cases appear whenever $p$ divides $2(n-1)$. This is also one of the reasons why non-symplectic involutions deserve to be discussed separately, since $2$ divides $2(n-1)$ for all $n \geq 2$.

If $X$ is of $\hskn$-type, an automorphism $\sigma \in \aut(X)$ is uniquely determined by the associated isometry $\sigma^* \in O(H^2(X, \IZ))$. In turn, it is possible to describe $\sigma^*$ by means of the invariant lattice $T = H^2(X, \IZ)^{\sigma^*}$ and its orthogonal complement $S = T^ \perp \subset H^2(X, \IZ)$. As a lattice, $H^2(X, \IZ)$ is isometric to the abstract lattice $L \coloneqq U^{\oplus 3} \oplus E_8^{\oplus 2} \oplus \langle -2(n-1) \rangle$, therefore we can fix an isometry $\eta: H^2(X, \IZ) \ra L$ and consider $T$ and $S$ as orthogonal primitive sublattices of $L$. After recalling, in Section \ref{section: preliminaries}, some well-known results which we use throughout the paper and fixing the notation, in Section \ref{section: isometries} we study the properties of the pair $(T, S)$ in the non-symplectic case. The isometry classes of the two lattices depend on three numerical invariants of the automorphism: its order $p$ and two integers $m,a$ such that $\rk(S) = (p-1)m$ and $\frac{H^2(X, \IZ)}{T \oplus S} \cong \left( \frac{\IZ}{p \IZ} \right)^{\oplus a}$. We say that a triple $(p,m,a)$, with $p$ prime, is \emph{admissible} for a certain $n \geq 2$ if there exists a non-symplectic automorphism of order $p$ on a manifold of $\hskn$-type whose invariant and co-invariant lattices $T,S$ are as above. Using classical results in lattice theory, mainly by Nikulin \cite{nikulin}, it is possible to find a list of all admissible triples $(p,m,a)$ for each value of $n$. Our first main result, which is purely lattice-theoretic, concerns the classification of pairs $(T,S)$ corresponding to a given admissible triple.

\begin{theorem}\label{intro: thm adm triples}
If $p^2 \nmid 2(n-1)$, an admissible triple $(p,m,a)$ uniquely determines the lattice $S$, up to isometries. Its orthogonal complement $T \subset L$ is also uniquely determined (up to isometries of $L$) by $(p,m,a)$, if $O(S)\ra O(q_S)$ is surjective and $l(A_T) \leq 21 - (p-1)m$.
\end{theorem}
 
Notice that, if $p$ is odd, the first instance where $p^2 \mid 2(n-1)$ is for \mbox{$n=10, p=3$}. Combining Theorem \ref{intro: thm adm triples} with the previous study of the action of a non-symplectic automorphism on cohomology, we provide (in Subsection \ref{subsect: classif for n=3,4} and Appendix \ref{appendix: tables}) the complete list of admissible triples $(p,m,a)$ and of the corresponding (unique) pairs of lattices $(T,S)$ when $n=3$ and $n=4$, which are the cases of most immediate interest.

The remaining part of the paper is dedicated to constructing examples of non-symplectic automorphisms of odd prime order. For manifolds of $\hsk$-type, in \cite{bcs} the authors prove that \emph{natural} automorphisms of Hilbert schemes of points (which come from automorphisms of $K3$ surfaces) allow to realize all but a few admissible pairs $(T,S)$; the residual cases (except for the aforementioned automorphism of order $23$) are constructed as automorphisms of Fano varieties of lines on cubic fourfolds. For $n \geq 3$, it is necessary to expand our pool of tools. \emph{Induced} automorphisms on moduli spaces of (possibly twisted) sheaves on $K3$ surfaces, studied in \cite{mw} and \cite{ckkm}, directly generalize natural automorphisms and allow to realize many new pairs $(T,S)$: we show, in Section \ref{section: induced for n=3,4}, how to apply these constructions when $n=3,4$.

\begin{theorem}
For $n=3,4$, all admissible pairs of lattices $(T,S)$ with $\rk(T) \geq 2$ can be realized by natural or (possibly twisted) induced automorphisms. 
\end{theorem}

Admissible pairs $(T,S)$ where $T$ has rank one require special attention. There are only four distinct triples $(p,m,a)$ which determine pairs of lattices $(T,S)$ with $\rk(T) = 1$: two for $p=3$ and two for $p=23$. However, for a fixed $n$ at most two of them are admissible (no more than one for each value of $p \in \left\{ 3,23 \right\}$). We study these four cases in Proposition \ref{rkT=1bis}, providing the corresponding isometry classes of the pairs $(T,S)$: even though they can never be realized by natural or induced non-symplectic automorphisms, we prove the following result.

\begin{theorem}
Let $(p,m,a)$ be an admissible triple for a certain $n$ and $(T,S)$ the corresponding pair of lattices. If $\rk(T) = 1$,  there exists a manifold $X$ of $\hskn$-type and a non-symplectic automorphism $f \in \aut(X)$ of order $p$ with invariant lattice $T_f \cong T$ and $\left( T_f \right)^\perp \cong S$. 
\end{theorem}

The proof of this statement (Proposition \ref{prop: existence for rkT=1}) is not constructive, since it employs the Torelli theorem for IHS manifolds. However, in specific cases it is possible to provide a geometric construction of the automorphism. In Section \ref{subsection: autom llsvs} we focus on one of these pairs of lattices $(T,S)$ with $\rk(T) = 1$, corresponding to the admissible triple $(3,11,0)$ for $n=4$.

\begin{theorem}
The admissible pair of lattices $T = \langle 2 \rangle$, $S = U^{\oplus 2} \oplus E_8^{\oplus 2} \oplus A_2$ for $n=4$ is realized by a non-symplectic automorphism of order three on a ten-dimensional family of Lehn--Lehn--Sorger--van Straten eightfolds, obtained from an automorphism of the underlying family of cyclic cubic fourfolds.
\end{theorem}

We remark that this is the first known geometric construction of a non-induced, non-symplectic automorphism of odd order on a manifold of $\hskq$-type. Moreover, thanks to it we are able to complete the list of examples of automorphisms of odd prime order $p < 23$ which realize all admissible pairs $(T,S)$ for $n=3,4$.

\medskip
\textbf{Acknowledgements}. The authors thank Samuel Boissi\`ere, Andrea Cattaneo, Alice Garbagnati, Robert Laterveer and Giovanni Mongardi for many helpful discussions, as well as Christian Lehn and Manfred Lehn for their useful remarks and explanations. The authors are also extremely grateful to Alessandra Sarti and Bert van Geemen, for reading the paper and for their precious suggestions.

\section{Preliminary notions} \label{section: preliminaries}

\subsection{Lattices}\label{lattices}
\begin{defi}
A \emph{lattice} $L$ is a free abelian group endowed with a symmetric, non-degenerate bilinear form $(\cdot,\cdot): L \times L \ra \IZ$. The lattice is \emph{even} if the associated quadratic form is even on all elements of $L$.
If $t$ is a positive integer, $L(t)$ denotes the lattice having as bilinear form the one of $L$ multiplied by $t$. Examples of lattices, which we will often use, are the \emph{negative} definite lattices $A_h, E_r$ corresponding to the Dynkin diagrams of the same names, for $h \geq 1$ and $r \in \left\{ 6, 8 \right\}$. We also define the two following lattices:

\[ H_5 \coloneqq \begin{pmatrix}
2 & 1\\
1 & -2
\end{pmatrix}; \qquad 
K_{23} \coloneqq \begin{pmatrix}
-12 & 1\\
1 & -2
\end{pmatrix}.\]

The \emph{dual lattice} of $L$ is $L^\vee \coloneqq \Hom_ {\IZ}(L, \IZ)$, which admits the following description:
\[L^\vee = \left\{ u \in L \otimes \IQ : (u,v) \in \IZ \;\; \forall v \in L \right\}. \]
\end{defi}

Clearly, $L$ can be seen as a subgroup of $L^\vee$ of maximal rank, thus the quotient $A_L \coloneqq L^\vee / L$ is a finite group, called the \emph{discriminant group} of $L$. We denote by $\discr(L)$ the order of the discriminant group, while the \emph{length} $l(A_L)$ is defined as the minimal number of generators of $A_L$. If $A_L = \left\{ 0\right\}$, the lattice $L$ is said to be \emph{unimodular}: an example of unimodular lattice is the (unique) even hyperbolic lattice $U$ of rank two. If instead $A_L \cong \left( \frac{\IZ}{p\IZ}\right)^{\oplus k}$ for a prime number $p$ and a non-negative integer $k$, then the lattice $L$ is said to be \emph{$p$-elementary}: in this case, $l(A_L) = k$.

For an even lattice $L$, it is possible to define the \emph{discriminant quadratic form} $q_L: A_L \ra \IQ/2\IZ$ by extending $(\cdot, \cdot)$ to $L^\vee$ and then passing to the quotient (see for instance \cite[\S 1.2]{dolgachev}). If $A_L$ is a finite direct sum of cyclic groups $A_i$, we write $ \bigoplus_i A_i(\alpha_i)$ if the discriminant form $q_L$ takes value $\alpha_i \in \IQ/2\IZ$ on the generator of the summand $A_i$. We will sometimes use the following result.

\begin{prop}\label{nik 1.2.1} \cite[Proposition 1.2.1]{nikulin}.
 Let $A$ be an abelian group, $q$ a finite quadratic form on $A$ and $H \subset A$ a subgroup. If the restriction $q\vert_H$ is non-degenerate, then $q = q\vert_H \oplus q\vert_{H^\perp}$.
\end{prop}

A lattice isometry of $L$ induces in a natural way an isometry of $\left(A_L, q_L \right)$, as explained in \cite[\S 1.2]{dolgachev}: in this way it is possible to define a canonical homomorphism between the orthogonal groups $O(L) \ra O(q_L)$. We will denote by $\bar{\psi} \in O(q_L)$ the image of $\psi \in O(L)$ under this homomorphism; similarly, an isomorphism of lattices $\varphi: L_1 \ra L_2$ induces an isomorphism of discriminant forms $\bar{\varphi}: q_{L_1} \ra q_{L_2}$ (\cite[\S 1.4]{nikulin}).

The \emph{signature} of a lattice $L$ is the signature of the $\IR$-linear extension of the bilinear form $(\cdot, \cdot )$ to $L \otimes_{\IZ} \IR$; together with the discriminant quadratic form $q_L$, it defines the \emph{genus} of $L$ (see \cite[\S 1]{nikulin}).

\begin{theorem}\label{l<rk-2} \cite[Proposition 1.4.7]{dolgachev}, \cite[Theorem 2.2]{morrison}.
An even, indefinite lattice $L$ with $l(A_L) \leq \rk(L)-2$ is uniquely determined, up to isometries, by its signature and its discriminant form $q_L$; moreover, the natural homomorphism $O(L) \ra O(q_L)$ is surjective.
\end{theorem} 

By \cite[Theorem 1.3.1]{nikulin}, two even lattices $L_1$, $L_2$ have isomorphic discriminant forms $q_{L_1} \cong q_{L_2}$ if and only if there exist unimodular lattices $V_1$, $V_2$ such that $L_1 \oplus V_1 \cong L_2 \oplus V_2$. Moreover, by \cite[Theorem 1.1.1(a)]{nikulin} the signature $(v_{(+)}, v_{(-)})$ of a unimodular lattice $V$ satisfies $v_{(+)}-v_{(-)} \equiv 0$ (mod $8$). It is therefore possible to define the \emph{signature (mod $8$)} of a finite quadratic form $q$: $\signt(q) = l_{(+)} - l_{(-)}$ (mod $8$), where $(l_{(+)}, l_{(-)})$ is the signature of an even lattice $L$ such that $q_L = q$.

We adopt the notation of \cite{brieskorn}. Let $p$ be an odd prime; by \cite[Proposition 1.8.1]{nikulin}, there are only two non-isometric, non-degenerate discriminant forms on $\frac{\IZ}{p^\alpha \IZ}$ ($\alpha \geq 1$): they are denoted by $w^\epsilon_{p,\alpha}$, with $\epsilon \in \left\{ -1, +1\right\}$. The quadratic form $w^{+1}_{p,\alpha}$ has generator value $q(1) = \frac{a}{p^\alpha}$ (mod $2 \IZ$), where $a$ is the smallest positive even number which is a quadratic residue modulo $p$; instead, for $w^{-1}_{p,\alpha}$ we have $q(1) = \frac{a}{p^\alpha}$ with $a$ the smallest positive even number that is not a quadratic residue modulo $p$. Thus, a non-degenerate quadratic form $q$ on $\frac{\IZ}{p^\alpha \IZ}$ such that $q(1) = \frac{x}{p^\alpha}$ is isometric to $w^{\epsilon}_{p,\alpha}$ , with $\epsilon = (\frac{x}{p})$ (using Legendre symbol).

Any non-degenerate quadratic form on $\left( \frac{\IZ}{p \IZ}\right)^{\oplus k}$, $k \geq 1$, is isomorphic to a direct sum of forms of type $w^{+1}_{p,1}$ and $w^{-1}_{p,1}$, with $w^{+1}_{p,1} \oplus w^{+1}_{p,1} \cong w^{-1}_{p,1} \oplus w^{-1}_{p,1}$ (\cite[Proposition 1.8.2]{nikulin}). This means that, if $S$ is a $p$-elementary lattice with discriminant group of length $k$, the form $q_S$ on $A_S$ can only be of two types, up to isometries:
\begin{equation*}
 q_S = \begin{cases}
\left( w^{+1}_{p,1} \right)^{\oplus k}\\
\left( w^{+1}_{p,1} \right)^{\oplus k-1} \oplus w^{-1}_{p,1}\\
\end{cases}
\end{equation*}

\begin{rem}\label{forme su S}
The signatures (mod $8$) of the discriminant forms $w^{\epsilon}_{p,\alpha}$ were computed in \cite{wall} (see also \cite[Proposition 1.11.2]{nikulin}): in particular, $\signt(w^{+1}_{p,1}) \equiv 1-p$ (mod $8$) and $\signt(w^{-1}_{p,1}) \equiv 5-p$ (mod $8$). Therefore, if $S$ is $p$-elementary and $\signt(S) = (s_{(+)}, s_{(-)})$, the quadratic form on the discriminant group $A_S = \left( \frac{\IZ}{p \IZ}\right)^{\oplus k}$ is
\begin{equation}
 q_S = \begin{cases} \label{form q_S}
\left( w^{+1}_{p,1} \right)^{\oplus k}  & \text{if } s_{(+)}-s_{(-)} \equiv k(1-p) \text{ (mod } 8 \text{)}\\
\left( w^{+1}_{p,1} \right)^{\oplus k-1} \oplus w^{-1}_{p,1}  & \text{if } s_{(+)}-s_{(-)} \equiv k(1-p) +4 \text{ (mod } 8 \text{)}\\
\end{cases}
\end{equation}

This means that the quadratic form of a $p$-elementary lattice ($p \neq 2$) is uniquely determined by its signature (see \cite[\S 1]{rudakov-shafarevich} for additional details).
\end{rem}

We recall that a sublattice $M \subset L$ is \emph{primitive} if the quotient $L/M$ is free; analogously, an embedding of lattices $i: S \hookrightarrow L$ is primitive if $i(S)\subset L$ is a primitive sublattice.

\begin{defi}\label{def isomorph}
Two primitive embeddings $i:S \hookrightarrow M$, $j:S \hookrightarrow M'$  define \emph{isomorphic primitive sublattices} if there exists $\varphi: M \ra M'$ isomorphism such that $\varphi(i(S))=j(S)$. 
\end{defi}

The following fundamental result, proved by Nikulin in \cite[Proposition 1.15.1]{nikulin}, provides a characterization of primitive embeddings.

\begin{theorem}\label{nik embedd}
Let $S$ be an even lattice of signature $(s_{(+)}, s_{(-)})$ and discriminant form $q_S$. Primitive embeddings $i: S \hookrightarrow L$, for $L$ even lattice of invariants $(m_{(+)},m_{(-)},q_L)$, are determined by quintuples $\Theta_i \coloneqq (H_S,H_L,\gamma, T, \gamma_T)$ such that:
\begin{itemize}
\item $H_S$ is a subgroup of $A_S$, $H_L$ is a subgroup of $A_L$ and $\gamma: H_S \ra H_L$ is an isometry $q_{S}\rvert_{H_S} \cong q_{L}\rvert_{H_L}$;
\item $T$ is a lattice of signature $(m_{(+)}-s_{(+)}, m_{(-)}-s_{(-)})$ and discriminant form $q_T = ((-q_S)\oplus q_L)\rvert_{\Gamma^\perp/\Gamma}$, where $\Gamma \subset A_S \oplus A_L$ is the graph of $\gamma$ and $\Gamma^\perp$ is its orthogonal complement in $A_S \oplus A_L$ with respect to the form $(-q_S) \oplus q_L$ with values in $\IQ/\IZ$;
\item $\gamma_T \in O(q_T)$. 
\end{itemize}
In particular, $T$ is isomorphic to the orthogonal complement of $i(S)$ in $L$. Moreover, two quintuples $\Theta$ and $\Theta'$ define isomorphic primitive sublattices if and only if $\bar{\mu}(H_S) = H'_S$ for \mbox{$\mu \in O(S)$} and there exist $\phi \in O(q_L)$, $\nu: T \ra T'$ isomorphism such that $\gamma' \circ \bar{\mu} = \phi \circ \gamma$ and $\bar{\nu} \circ \gamma_T = \gamma'_{T'} \circ \bar{\nu}$.
\end{theorem}

\subsection{Monodromies and global Torelli theorem for manifolds of $\hskn$-type}\label{monodromies}
An \emph{irreducible holomorphic symplectic} (IHS) manifold is a complex, smooth, compact, K\"ahler manifold $X$ such that $H^{2,0}(X) = \mathbb{C}\omega_X$, with $\omega_X$ an everywhere non-degenerate two-form. Examples of IHS manifolds are provided by $K3$ surfaces and, for any $n \geq 2$, Hilbert schemes of $n$ points on them, as well as their deformations, which are known as {\it manifolds of $\hskn$-type}. If $X$ is an IHS manifold, the second cohomology group $H^2(X, \IZ)$ admits a lattice structure, using the non-degenerate bilinear form of signature $(3, b_2(X) - 3)$ due to Beauville--Bogomolov--Fujiki (see \cite{beauville}). An automorphism $\sigma \in \aut(X)$ of prime order $p$ is \emph{non-symplectic} if $\sigma^*\omega_X = \xi \omega_X$, where $\xi$ is a primitive $p$-root of the unity. By \cite[\S 4]{beauville_rmks}, the existence of a non-symplectic automorphism on $X$ guarantees that $X$ is projective.   

The global Torelli theorem for IHS manifolds (\cite{huybr_torelli}, \cite{verb}) admits the following Hodge theoretic formulation, due to Markman (\cite[Theorem 1.3]{markman}).
 
\begin{theorem}[Global Torelli theorem] \label{torelli}
Let $X$, $Y$ be irreducible holomorphic symplectic manifolds. Then, $X$ and $Y$ are bimeromorphic if and only if there exists $f: H^2(X,\IZ) \ra H^2(Y,\IZ)$ isomorphism of integral Hodge structures and parallel transport operator. If, moreover, $f$ maps a K\"ahler class of $X$ to a K\"ahler class of $Y$, then there exists an isomorphism $\sigma:Y \ra X$ acting as $f$ on cohomology.
\end{theorem}

A definition of parallel transport operator can be found in \cite[Definition 1.1]{markman}. In the case $X=Y$, parallel transport operators $f: H^2(X,\IZ) \ra H^2(X,\IZ)$ are called \emph{monodromy operators}: they form a subgroup $\Mo^2(X) \subset O(H^2(X,\IZ))$. If $(X, \eta)$ is a marked holomorphic symplectic manifold, with $\eta: H^2(X,\IZ) \ra L$ lattice isomorphism, we denote by $\Mo^2(L) \coloneqq \eta \circ \Mo^2(X) \circ \eta^{-1} \subset O(L)$ the monodromy group of $L$: it is an arithmetic subgroup of $O(L)$ and it is independent on the choice of the marking, inside a connected component of the moduli space of marked pairs $(X, \eta)$ (see \cite[\S 9]{markman}).

If $X$ is a manifold of $\hskn$-type, $n \geq 2$, the second cohomology lattice $H^2(X, \IZ)$ has rank $23$ and it is isometric to $L = U^{\oplus 3}\oplus E_8^{\oplus 2} \oplus \langle-2(n-1)\rangle$. In this case, we have a very explicit description of monodromy operators on $L$. Let $\mathcal{N}$ be the subgroup of $O(L)$ generated by reflections with respect to classes of square $-2$ and by the negative of reflections with respect to $+2$-classes. Then, combining results of Markman (\cite[Theorem 1.2]{markman2}) and Kneser (\cite{kneser}), we obtain the following description.

\begin{theorem} \label{monodromy group} $\Mo^2(L) = \mathcal{N} = \left\{ g \in O(L) \mid \bar{g} = \pm \id_{A_L}, \spinn^L(g) = 1 \right\}.$
\end{theorem}

In the statement of the theorem, as usual, $\bar{g}$ is the isometry induced by $g$ on the discriminant group, while $\spinn$ denotes the \emph{real spinor norm}: for an even lattice $G$, we define $\spinn^G: O(G_\mathbb{R}) \ra \mathbb{R}^*/\left(\mathbb{R}^*\right)^2 \cong \left\{ \pm 1\right\}$ as
 \[ \spinn^G(\gamma) = \left( -\frac{v_1^2}{2}\right) \ldots \left( -\frac{v_r^2}{2}\right)\]
 \noindent if $\gamma = \rho_{v_1} \circ \ldots \circ \rho_{v_r}$ as a product of reflections with respect to vectors $v_i \in G_\mathbb{R}$ (in particular, by Cartan-Dieudonn\'e theorem, $r \leq \rk(G)$).

\section{Isometries induced by automorphisms of odd prime order}\label{section: isometries}

The aim of this section is to study the action of non-symplectic automorphisms of odd prime order on the second cohomology lattice of manifolds of $\hskn$-type. We will focus our attention on determining the properties of the invariant sublattice and of its orthogonal and we will show how to classify, for any $n$, their isometry classes by use of numerical parameters related to their signatures and lengths. This classification is explicitly discussed for $n=3,4$ in \S \ref{subsect: classif for n=3,4}. Moreover, in \S \ref{subsection: rkT=1} we study in greater depth the cases where the invariant lattice has rank one.

\subsection{Discriminant groups of invariant and co-invariant sublattices}\label{subsection: induced isom}
Let $X$ be a manifold of $\hskn$-type with an action of a finite group $G = \langle \sigma \rangle$, where $\sigma$ is a  non-symplectic automorphism of prime order $3 \leq p \leq 23$. The group $G$ acts by pullback on $H^2(X, \IZ) \cong L \coloneqq U^{\oplus 3}\oplus E_8^{\oplus 2} \oplus \langle-2(n-1)\rangle$; following the notation of \cite{smith}, we will denote by $T = T_G(X) \coloneqq H^2(X, \IZ)^{\sigma^*}$ the invariant sublattice of $H^2(X,\IZ)$  and by $S = S_G(X) \coloneqq T^\perp$ its orthogonal complement (the co-invariant lattice, as we will refer to it): they are both primitive.

\begin{rem}\label{equiv pairs S,T}
If we choose a marking $\eta: H^2(X, \IZ) \ra L$, then the invariant and co-invariant lattices of an automorphism of $X$ can also be regarded as primitive sublattices $T,S \subset L$. We point out that a different marking $\eta'$ will produce a pair of sublattices $(T',S')$ of $L$ which is isomorphic to $(T,S)$ in the sense of Definition \ref{def isomorph}. For this reason, we are interested in classifying the pairs $(T,S)$ only up to isomorphisms of primitive sublattices in $L$.
\end{rem}

We collect in the next proposition several results proved by Boissi\`ere--Nieper-Wi{\ss}kirchen--Sarti \cite{smith} and Tari \cite{tari}.

\begin{prop}\label{preliminaries}
Let $X$ be a manifold of $\hskn$-type and $G = \langle \sigma \rangle$ a group of prime order $p$ acting non-symplectically on $X$. Then:
\begin{itemize}
\item There exists a positive integer $m := m_G(X)$ such that $\rk(S) = (p-1)m$;
\item S has signature $(2, (p-1)m -2)$ and $T$ has signature $(1, 22 - (p-1)m)$;
\item $\frac{H^2(X,\IZ)}{T \oplus S}$ is a $p$-torsion group, i.e. $\frac{H^2(X,\IZ)}{T \oplus S} \cong \left(\frac{\IZ}{p \IZ} \right)^{\oplus a}$ for some non-negative integer $a := a_G(X)$;
\item $a \leq m$.
\end{itemize}
\end{prop} 

If we consider $T$ and $S$ as sublattices of $L$, then $T \oplus S$ is a  sublattice of maximal rank. The sequence of inclusions
\[ T \oplus S \subset L \subset L^\vee \subset (T \oplus S)^\vee \cong T^\vee \oplus S^\vee \]
\noindent provides an identification of $\frac{L}{T \oplus S} \cong \left(\frac{\IZ}{p \IZ} \right)^{\oplus a}$ with a subgroup $M \subset A_T \oplus A_S$: a direct computation shows that $M$ is isotropic and $M^\perp / M \cong A_L$ (see \cite[\S 5]{nikulin}). Denoting by $p_T$ and $p_S$ the two projections from $A_T \oplus A_S$ to $A_T$ and $A_S$ respectively, their restrictions to $M$ are injective (because $T \hookrightarrow L$ and $S \hookrightarrow L$ are primitive; see again \cite[\S 5]{nikulin}): the isomorphic images are $M_T \coloneqq p_T(M) \subset A_T$ and $M_S \coloneqq p_S(M) \subset A_S$. Since the discriminant groups are finite, we conclude $p^a \mid \discr(T)$, $p^a \mid \discr(S)$. Moreover, the homomorphism $\gamma := p_S \circ (p_T)^{-1}\rvert_{M_T}: M_T \ra M_S$ is an anti-isometry, as a consequence of the isotropy of $M$: this means that $q_T\rvert_{M_T} \cong - q_S\rvert_{M_S}$.

\begin{lemma}\label{S p-elem}
Let $\psi \in \Mo^2(L)$ be the monodromy induced on $L$ by a non-symplectic automorphism $\sigma$ of prime order $p \geq 3$ of a manifold of $\hskn$-type. Then: 
\begin{enumerate}
\item[\textit{(i)}] the action of $\psi$ on $M^\perp \subset A_T \oplus A_S$ is trivial;
\item[\textit{(ii)}] the co-invariant lattice $S = \left( L^\psi \right)^\perp$ is $p$-elementary.
\end{enumerate}
\end{lemma}

\begin{proof}
\begin{enumerate}
\item[\textit{(i)}]The isometry $\psi$ acts trivially on $A_L \cong M^\perp/M$, from Theorem \ref{monodromy group} and because $\psi^p = \id$ with $p$ odd (therefore $\psi$ cannot induce $-\id$ on the discriminant $A_L$). This implies that, for any element $(x,y) \in M^\perp \subset A_T \oplus A_S$, we have $\bar{\psi}(x,y) - (x,y) \in M$; moreover, $\psi$ acts trivially on the discriminant group $A_T$ (because $\psi\vert_T = \id_T$), thus $\bar{\psi}(x,y) - (x,y) = (0, \bar{\psi}(y)-y)$ (the natural inclusions of $A_T$ and $A_S$ in $A_{T \oplus S} \cong A_T \oplus A_S$ are $\psi$-equivariant). Since $M$ is the graph in $A_T \oplus A_S$ of the anti-isometry $\gamma : M_T \ra M_S$ we deduce that $\bar{\psi}(y)=y$ for any $y \in p_S(M^\perp)$. This means that the action of $\psi$ is trivial on $M^\perp$, not only on the quotient $M^\perp/M$. 

\item[\textit{(ii)}] For any $n \geq 2$, the lattice $L = U^{\oplus 3}\oplus E_8^{\oplus 2} \oplus \langle-2(n-1)\rangle$ admits a natural primitive embedding inside the \emph{Mukai lattice} $\Lambda_{24} \coloneqq U^{\oplus 4}\oplus E_8^{\oplus 2}$ (see \cite[Corollary 9.5]{markman}). As we remarked in the previous point of the proof, the action of $\psi$ on the discriminant $A_L$ is trivial: this allows us to extend $\psi$ to an isometry $\rho \in O(\Lambda_{24})$ such that $\rho\vert_{L^\perp} = \id$, by \cite[Corollary 1.5.2]{nikulin}. The lattice $\Lambda_{24}$ is unimodular, therefore both the invariant lattice \mbox{$T_\rho = \Lambda_{24}^\rho \subset \Lambda_{24}$} and the co-invariant $S_\rho = (T_\rho)^\perp \subset \Lambda_{24}$ are $p$-elementary (see for instance \cite[Lemme 2.10]{tari}). Since $L^\perp \subset T_\rho$, passing to the orthogonal complements we have $S_\rho \subset L$, and therefore $S = S_\rho$ is $p$-elementary.\qedhere
\end{enumerate}
\end{proof}

For fixed values of $n \geq 2$ and $p \geq 3$ prime, we write $2(n-1) = p^\alpha \beta$ with $\alpha, \beta$ integers, $\alpha \geq 0$ and $(p, \beta) = 1$. Then $A_L \cong \frac{\IZ}{2(n-1)\IZ} \cong \frac{\IZ}{p^\alpha \IZ} \oplus \frac{\IZ}{\beta \IZ}$ is an orthogonal splitting (see \cite[Proposition 1.2.2]{nikulin}): in particular, we can show that there exists a subgroup of $A_T$ isometric to the summand $\frac{\IZ}{\beta \IZ}$.

\begin{lemma}\label{sylow structure}
Let $(A_T)_p$ and $(A_S)_p$ be the Sylow $p$-subgroups of $A_T$ and $A_S$ respectively; then
\[A_T = (A_T)_p \oplus \frac{\IZ}{\beta \IZ}, \qquad A_S = (A_S)_p.\]
Moreover, $\left| A_T \right|= p^a \beta t$ and $\left| A_S \right|= p^a s$ for some positive integers $t,s$ such that $ts = p^\alpha$.
\end{lemma}

\begin{proof}
Since $A_L \cong M^\perp/M$ and $\left| M \right| = p^a$, we deduce $\left| M^\perp \right| = p^{a + \alpha}\beta$. Moreover, $(p,\beta) = 1$, thus there exists a unique subgroup $N \subset M^\perp$ of order $\beta$, such that the restriction to $N$ of the projection $M^\perp \ra M^\perp/M$ is injective. Using the fact that there is also a unique subgroup of order $\beta$ inside $A_L$, we conclude that $N$ is isomorphic to the component $\frac{\IZ}{\beta \IZ}$ of $A_L$. By Lemma \ref{S p-elem}, the action of the automorphism $\sigma$ on $N \subset M^\perp$ is trivial and any element of $p_S(N)$ is of $p$-torsion: we are lead to conclude $p_S(N) = 0$, because $(p,\beta)=1$. Thus, $N$ is contained in $A_T$.

Since $M_T \subset (A_T)_p$, $M_S \subset (A_S)_p$ we can write $\left| A_T \right|= p^a \beta t$ and $\left| A_S \right|= p^a s$, with $t,s$ positive integers. From
\[\left[ L : (T \oplus S)\right]^2 = \frac{\discr(T) \cdot \discr(S)}{\discr(L)} = \frac{\left| A_T \right| \left|A_S \right|}{\left| A_L \right|}\]
\noindent we get $ts = p^\alpha$. The two integers $t,s$ are therefore non-negative powers of $p$.
\end{proof}

We are now ready to describe the structures of the two discriminant groups $A_T$ and $A_S$.

\begin{prop}\label{discriminant groups structure}
Let $X$ be a manifold of $\hskn$-type and $G = \langle \sigma \rangle$ a group of odd prime order $p$ acting non-symplectically on $X$. Then one of the following cases holds:
\begin{enumerate}
\item[\textit{(i)}] $A_S = M_S \cong \left(\frac{\IZ}{p \IZ}\right)^{\oplus a}$, $A_T \cong M_T \oplus A_L \cong \left(\frac{\IZ}{p \IZ}\right)^{\oplus a} \oplus \frac{\IZ}{p^\alpha \IZ} \oplus \frac{\IZ}{\beta \IZ}$;
\item[\textit{(ii)}] $\alpha = 1$, $a=0$, $A_S \cong \frac{\IZ}{p \IZ}$, $A_T \cong \frac{\IZ}{\beta \IZ}$;
\item[\textit{(iii)}] $\alpha \geq 1$, $a \geq 1$, $A_S \cong \left(\frac{\IZ}{p \IZ}\right)^{\oplus a+1}$, $A_T \cong \left(\frac{\IZ}{p \IZ}\right)^{\oplus a-1} \oplus \frac{\IZ}{p^\alpha \IZ} \oplus \frac{\IZ}{\beta \IZ}$.
\end{enumerate}
\end{prop}

\begin{proof}
If $a=0$, the group $M$ is trivial and $A_L \cong A_T \oplus A_S$. By Lemma \ref{sylow structure} we deduce that there are only two possibilities: $A_S = 0$, $A_T \cong A_L$ or $A_S \cong \frac{\IZ}{p^\alpha \IZ}$, $A_T \cong \frac{\IZ}{\beta \IZ}$. The second case, though, is admissible only for $\alpha = 1$, because we know that $S$ is $p$-elementary by Lemma \ref{S p-elem}.

From now on we will assume $a \geq 1$. Let us first consider the case $\alpha = 0$: this implies $\beta = 2(n-1)$, $t=s=1$. Then, using Lemma \ref{sylow structure}, we conclude $A_S = M_S$ and $A_T = M_T \oplus \frac{\IZ}{\beta \IZ} \cong M_T \oplus A_L$.

If $\alpha = 1$ we have $2(n-1) = p \beta$ and $ts = p$. There are two possibilities:
\begin{itemize}
\item $t=p, s=1$. In this case, $A_S = M_S \cong \left(\frac{\IZ}{p \IZ}\right)^{\oplus a}$; in particular, $q_S\rvert_{M_S} = q_S$ is non-degenerate and the same holds for $q_T\rvert_{M_T}$, since $q_T\rvert_{M_T} \cong - q_S\rvert_{M_S}$. Then, by Proposition \ref{nik 1.2.1} we can write $A_T = M_T \oplus M_T^\perp$, which implies $(A_T)_p \cong \left(\frac{\IZ}{p \IZ}\right)^{\oplus a+1}$ and therefore $A_T \cong M_T \oplus A_L$.
\item $t=1, s=p$. Now $A_T = M_T \oplus \frac{\IZ}{\beta \IZ}$, since $(A_T)_p = M_T$. Hence $q_T\rvert_{M_T}$ is non-degenerate, which again implies that also $q_S\rvert_{M_S}$ is non-degenerate, i.e.\ $A_S = M_S \oplus M_S^\perp$. We are lead to conclude $A_S = M_S \oplus \frac{\IZ}{p \IZ} \cong \left(\frac{\IZ}{p \IZ}\right)^{\oplus a+1}$.
\end{itemize}

\noindent Therefore, if $\alpha=1$ (and $a \geq 1$) both cases (i), (iii) appearing in the statement are admissible.

Now assume $\alpha \geq 2$: set $H := (A_T)_p \oplus A_S \subset A_T \oplus A_S$ and let $H[p] \subset H$ be the $p$-torsion subgroup. Since $M^\perp/M \cong A_L \cong \frac{\IZ}{p^\alpha \IZ} \oplus \frac{\IZ}{\beta \IZ}$, there exists an element $x \in H$ of order at least $p^\alpha$: the quotient $\langle x \rangle / (\langle x \rangle \cap H[p])$ has then order at least $p^{\alpha-1}$, which shows that $\left[ H : H[p]\right] \geq p^{\alpha - 1}$. On the other hand, $\left[ H : H[p]\right] \leq p^\alpha$: in fact, $\lvert H[p] \rvert \geq p^{2a}$, because $M_T \oplus M_S \subset H[p]$, and $\lvert H \rvert = p^at \cdot p^as = p^{2a + \alpha}$ (by Lemma \ref{sylow structure}). We conclude that the index $\left[ H : H[p]\right]$ is either $p^\alpha$ or $p^{\alpha-1}$.

If $\left[ H : H[p]\right] = p^\alpha$, then $H[p] = M_T \oplus M_S$. By construction $H = H_p$, therefore:
\begin{equation}\label{general form of H}
H \cong \bigoplus_{i=1}^{2a + \alpha} \left( \frac{\IZ}{p^i \IZ}\right)^{\oplus m_i}, 	\qquad H[p] \cong \bigoplus_{i=1}^{2a + \alpha} \left( \frac{\IZ}{p \IZ}\right)^{\oplus m_i}.
\end{equation}
\noindent for suitable integers $m_i \geq 0$ such that $\sum_{i}im_i = 2a + \alpha$ and $\sum_{i} m_i = 2a$. Thus, the coefficients $m_i$ must satisfy $\alpha = \sum_{i}(i-1)m_i$; moreover, since we know that $H$ contains an element of order at least $p^\alpha$, there exists $j \geq \alpha$ such that $m_j \geq 1$. This leaves us with two possibilities for the choice of the coefficients $m_i$.
\begin{itemize}
\item $H \cong \left( \frac{\IZ}{p \IZ}\right)^{\oplus 2a-1} \oplus \frac{\IZ}{p^{\alpha+1} \IZ}$. Then, we have either $A_S \cong \left( \frac{\IZ}{p \IZ}\right)^{\oplus a-1} \oplus \frac{\IZ}{p^{\alpha+1} \IZ}$, \mbox{$\left(A_T\right)_p = M_T$} or $A_S = M_S, \left(A_T\right)_p \cong \left( \frac{\IZ}{p \IZ}\right)^{\oplus a-1} \oplus \frac{\IZ}{p^{\alpha+1} \IZ}$. Both cases are not admissible: by Proposition \ref{nik 1.2.1} (as we remarked discussing $\alpha = 1$) we would need to be able to write, respectively, $A_S = M_S \oplus M_S^\perp$ and $A_T = M_T \oplus M_T^\perp$, but now this is not possible.

\item $H \cong \left( \frac{\IZ}{p \IZ}\right)^{\oplus 2a-2} \oplus \frac{\IZ}{p^{2} \IZ} \oplus \frac{\IZ}{p^{\alpha} \IZ}$. Disregarding the cases where $A_S = M_S$ or $\left( A_T \right)_p = M_T$ (which can be excluded as in the previous point) we are left with two alternatives:
\begin{itemize}
\item $\left(A_T\right)_p \cong \left( \frac{\IZ}{p \IZ}\right)^{\oplus a-1} \oplus \frac{\IZ}{p^\alpha \IZ}$, $A_S \cong \left( \frac{\IZ}{p \IZ}\right)^{\oplus a-1}\oplus \frac{\IZ}{p^2 \IZ}$;
\item $\left(A_T\right)_p \cong \left( \frac{\IZ}{p \IZ}\right)^{\oplus a-1} \oplus \frac{\IZ}{p^2 \IZ}$, $A_S \cong \left( \frac{\IZ}{p \IZ}\right)^{\oplus a-1}\oplus \frac{\IZ}{p^\alpha \IZ}$.
\end{itemize}
In both cases, though, the lattice $S$ is not $p$-elementary, contradicting Lemma \ref{S p-elem}.
\end{itemize} 

Now assume $\left[ H : H[p]\right] = p^{\alpha-1}$. We can again write $H$ and $H[p]$ as in (\ref{general form of H}), where now $\sum_{i}im_i = 2a + \alpha$, $\sum_{i} m_i = 2a+1$ and as before there exists $j \geq \alpha$ such that $m_j \geq 1$. We then deduce $H \cong \left( \frac{\IZ}{p \IZ}\right)^{\oplus 2a} \oplus \frac{\IZ}{p^\alpha \IZ}$, which gives rise to four possible conclusions:
\begin{itemize}
\item $\left(A_T\right)_p \cong \left( \frac{\IZ}{p \IZ}\right)^{\oplus a} \oplus \frac{\IZ}{p^\alpha \IZ}$, $A_S \cong \left( \frac{\IZ}{p \IZ}\right)^{\oplus a}$, meaning $A_T \cong M_T \oplus A_L$ and $A_S = M_S$;
\item $\left(A_T\right)_p \cong \left( \frac{\IZ}{p \IZ}\right)^{\oplus a-1} \oplus \frac{\IZ}{p^\alpha \IZ}$, $A_S \cong \left( \frac{\IZ}{p \IZ}\right)^{\oplus a+1}$;
\item $\left(A_T\right)_p \cong \left( \frac{\IZ}{p \IZ}\right)^{\oplus a}$, $A_S \cong \left( \frac{\IZ}{p \IZ}\right)^{\oplus a} \oplus \frac{\IZ}{p^\alpha \IZ}$;
\item $\left(A_T\right)_p \cong \left( \frac{\IZ}{p \IZ}\right)^{\oplus a+1}$, $A_S \cong \left( \frac{\IZ}{p \IZ}\right)^{\oplus a-1} \oplus \frac{\IZ}{p^\alpha \IZ}$.
\end{itemize}

The last two cases are excluded because $S$ is $p$-elementary by Lemma \ref{S p-elem}.
\end{proof}

\begin{rem}\label{parity}
We can make some additional remarks on the structures of the discriminant groups $A_T, A_S$ recalling the following result.

\begin{theorem} \label{thm: tari} \cite[Th\'eor\`eme 2.23]{tari}
Let $M$ be a lattice and $\psi \in O(M)$ of prime order $p \neq 2$. Then $p^m \discr (S_\psi)$ is a square in $\IZ$, where $m = \frac{\rk(S_\psi)}{p-1}$.
\end{theorem}

Let $X$ be a manifold of $\hskn$-type and $\psi \in \Mo^2(L)$ the isometry induced on $L$ by an automorphism $\sigma \in \aut(X)$ of prime order $p \geq 3$. From Proposition \ref{discriminant groups structure} we know that $\discr(S) = \lvert A_S\rvert$ is either $p^a$ or $p^{a+1}$. In particular:
\begin{itemize}
\item if $p \nmid 2(n-1)$ (i.e.\ $\alpha = 0$), the groups $A_T$, $A_S$ are as in Proposition \ref{discriminant groups structure}, case (i), therefore $a$ and $m$ must be of same parity by Theorem \ref{thm: tari}. 
\item If $p \mid 2(n-1)$ (i.e.\ $\alpha \geq 1$), $a$ and $m$ are not required to have same parity: the structures of $A_T$ and $A_S$ are the ones given in Proposition \ref{discriminant groups structure} case (i) if $a$ and $m$ have same parity, the ones of cases (ii) or (iii) if $a$ and $m$ have different parity. 
\end{itemize} 
\end{rem}

\subsection{Admissible triples}\label{subsection: adm triples}

We are now interested in studying primitive embeddings of lattices $T, S \hookrightarrow L$ satisfying Proposition \ref{preliminaries} and Proposition \ref{discriminant groups structure}, assuming $p \geq 3$. For the purposes of this work, we restrict to $\alpha \leq 1$: notice that, since $2(n-1) = p^\alpha \beta$, the first instance with $\alpha \geq 2$ occurs for $n=10$, i.e.\ on manifolds of dimension $20$.

Our main result is Theorem \ref{thm: admissible triples}, in which we show that, under suitable hypotheses, starting from the values $(p,m,a)$ defined in Proposition \ref{preliminaries} it is possible to uniquely determine the isometry classes of $T$ and $S$. To do so we first need to provide a characterization of primitive embeddings $S \hookrightarrow L$ for lattices $S$ of this kind (Lemma \ref{embeddings of S} and Proposition \ref{unicity embedding and Tbis}). Finally, in Proposition \ref{prop: forms q_S and q_T} we describe all possible structures, up to isometries, for the discriminant quadratic forms $q_S$ and $q_T$.

We recall that, by Proposition \ref{preliminaries}, the lattice $S$ has signature $(2, (p-1)m - 2)$; moreover, by Lemma \ref{S p-elem} it is $p$-elementary, with discriminant $A_S = \left( \frac{\IZ}{p \IZ}\right)^{\oplus k}$, where $k$ is the length of $A_S$. Then (see Remark \ref{forme su S}) there are only two non-isometric possible forms $q_S$, the ones in (\ref{form q_S}).

Since $L=U^{\oplus 3}\oplus E_8^{\oplus 2} \oplus \langle-2(n-1)\rangle$, the quadratic form $q_L$ on $A_L = \frac{\IZ}{2(n-1)\IZ}$ is such that $q_L(1) = -\frac{1}{2(n-1)}$ (mod $2\IZ$). If we write $2(n-1)=p^\alpha \beta$, with $(p, \beta)=1$, then a trivial computation shows that:
\[ q_L = \frac{\IZ}{2(n-1)\IZ}\left( -\frac{1}{2(n-1)} \right) \cong \frac{\IZ}{p^\alpha\IZ} \left( -\frac{\beta}{p^\alpha} \right) \oplus \frac{\IZ}{\beta \IZ} \left( -\frac{p^\alpha}{\beta} \right)  .\]

Denoting by $q_{\alpha, \beta}$ the quadratic form $\frac{\IZ}{\beta \IZ}\left(-\frac{p^\alpha}{\beta}\right)$, we conclude:
\begin{equation}\label{form q_L}
q_L = \begin{cases}
w^{+1}_{p,\alpha} \oplus q_{\alpha, \beta} & \text{if } \left( \frac{-\beta}{p} \right) = +1\\
w^{-1}_{p,\alpha} \oplus q_{\alpha, \beta} & \text{if } \left( \frac{-\beta}{p} \right) = -1\\
\end{cases}
\end{equation} 

\begin{lemma}\label{embeddings of S}
Let $S$ be an even lattice with discriminant group $A_S = \left( \frac{\IZ}{p \IZ}\right)^{\oplus k}$ and invariants $(2, (p-1)m - 2, q_S)$. Let $L=U^{\oplus 3} \oplus E_8^{\oplus 2} \oplus \langle -2(n-1) \rangle$ and let $e \in A_L$ be the generator of the component $\frac{\IZ}{p^\alpha\IZ}$ of $A_L \cong \frac{\IZ}{p^\alpha\IZ} \left( -\frac{\beta}{p^\alpha} \right) \oplus \frac{\IZ}{\beta \IZ} \left( -\frac{p^\alpha}{\beta} \right)$. Then:
\begin{enumerate}
\item[\textit{(i)}] If $\alpha = 0$, primitive embeddings of $S$ in $L$ compatible with Proposition \ref{discriminant groups structure} are determined by pairs $(T, \gamma_T)$, with $T$ a lattice of signature \mbox{$(1, 22 - (p-1)m)$}, $q_T = (-q_S)\oplus q_L$ and $\gamma_T \in O(q_T)$; two pairs $(T, \gamma_T)$ and $(T', \gamma'_{T'})$ determine isomorphic sublattices in $L$ if and only if there exists an isometry $\nu: T \ra T'$ such that $\bar{\nu} \circ \gamma_T = \gamma'_{T'} \circ \bar{\nu}$.
\item[\textit{(ii)}] If $\alpha = 1$, primitive embeddings of $S$ in $L$ compatible with Proposition \ref{discriminant groups structure} are determined by triples $(x, T, \gamma_T)$, with $T$ of signature \mbox{$(1, 22 - (p-1)m)$}, $\gamma_T \in O(q_T)$ and either:
\begin{enumerate}
\item $x = 0$, $q_T = (-q_S)\oplus q_L$, or
\item $x \in A_S[p]$ with $q_S(x)= -\frac{\beta}{p} \; (\text{mod } 2\IZ)$ and $\Gamma^\perp/\Gamma \cong \left( \frac{\IZ}{p \IZ}\right)^{\oplus k-1} \oplus \frac{\IZ}{\beta \IZ}$, where $\Gamma \subset A_S \oplus A_L$ is the subgroup generated by $(x,e)$ and $\Gamma^\perp$ is its orthogonal complement with respect to the form $(-q_S)\oplus q_L$; moreover, \mbox{$q_T = ((-q_S)\oplus q_L)\rvert_{\Gamma^\perp/\Gamma}$}.
\end{enumerate}
Two triples $(x, T, \gamma_T)$ and $(x', T', \gamma'_{T'})$ determine isomorphic sublattices in $L$ if and only if there exists $\mu \in O(S)$ and an isometry $\nu: T \ra T'$, such that $\bar{\mu}(x)=x'$ and $\bar{\nu} \circ\gamma_T = \gamma'_{T'} \circ \bar{\nu}$.
\end{enumerate}
\end{lemma}

\begin{proof}
Each primitive embedding $i: S \hookrightarrow L$ is determined by a quintuple $\Theta_i = (H_S, H_L, \gamma, T, \gamma_T)$ as in Theorem \ref{nik embedd}. Recalling that $T$ is the orthogonal complement of $i(S)$ in $L$, we ask  $\signt(T) = (1, 22 - (p-1)m)$. We will discuss separately the cases $\alpha= 0$ and $\alpha = 1$.
\begin{enumerate}
\item[\textit{(i)}] $\alpha = 0$. Since $p$ and $\beta$ are coprime, the only possibility is $H_S = \left\{ 0 \right\}$, $H_L = \left\{ 0 \right\}$ and $\gamma = \id$. The embedding $S \hookrightarrow L$ is therefore determined by the pair $(T, \gamma_T)$. In particular, we have: $\Gamma = \left\{ (0,0) \right\}$, $\Gamma^\perp = A_S \oplus A_L$, which implies $A_T = A_S \oplus A_L$ and $q_T = (-q_S) \oplus q_L$. This is coherent with case (i) of Proposition \ref{discriminant groups structure}.

\item[\textit{(ii)}] $\alpha = 1$. We have again the case $H_S = \left\{ 0 \right\}, H_L = \left\{ 0 \right\}$, $\gamma = \id$ (which means that $S$ and $T$ are as in case (i) of Proposition \ref{discriminant groups structure}, hence $l(A_T) = k+1$). This case corresponds to the triples where $x=0$ and it is described as for $\alpha = 0$.

Alternatively, provided that there exists an element $x \in A_S$ of order $p$ such that $q_S(x)= q_L(e)$, with $e$ as in the statement, we can also take $H_S = \langle x\rangle, H_L = \langle e \rangle$, $\gamma: x \mapsto e$. Notice that such an element $x$ does not exist only if $k=0$ or if $q_S =w^{\xi}_{p,1}, q_L=w^{-\xi}_{p,1} \oplus q_{1, \beta}$, with $\xi \in \left\{ \pm 1\right\}$: in all other cases, using the isomorphism $w^{+1}_{p,1} \oplus w^{+1}_{p,1} \cong w^{-1}_{p,1} \oplus w^{-1}_{p,1}$, we can write the form $q_S$ as in (\ref{form q_S}), where at least one of the direct summands is of the same type of the $w^{\epsilon}_{p,1}$ appearing in $q_L$ (the component corresponding to the subgroup $H_L$).

In this setting, the graph $\Gamma$ of $\gamma$ is the subgroup of $A_S \oplus A_L$ generated by $(x,e)$; in particular, since $\Gamma \cong \frac{\IZ}{p \IZ}$, the quotient $\Gamma^\perp/\Gamma$ cannot be isomorphic to $A_S \oplus A_L$, meaning that we are not in case (i) of Proposition \ref{discriminant groups structure}. Nevertheless, if $\alpha = 1$ and $k \geq 1$ the structures of the discriminant groups can also be as in cases (ii) or (iii), where $l(A_T) = \max\{1,k-1\}$: the embedding is admissible if $\Gamma^\perp/\Gamma \cong \left( \frac{\IZ}{p \IZ}\right)^{\oplus k-1} \oplus \frac{\IZ}{\beta \IZ}$, and -- if so -- the quadratic form on $A_T$ is $q_T = ((-q_S)\oplus q_L)\rvert_{\Gamma^\perp/\Gamma}$.
\end{enumerate} 

Finally, for both values of $\alpha$ the stated results about isomorphic sublattices follow directly from Theorem \ref{nik embedd}.
\end{proof}

Lemma \ref{embeddings of S} allows us to list all possible primitive embeddings $i: S \hookrightarrow L$ satisfying Proposition \ref{discriminant groups structure} for a given lattice $S$. We now prove that, adding some extra hypotheses, the number of distinct isometry classes for $i(S)^\perp$ is actually very limited.

\begin{prop}\label{unicity embedding and Tbis}
Let $S$ and $L$ be as in Lemma \ref{embeddings of S}, $\alpha \leq 1$ and $k \leq 21 - \alpha - (p-1)m$.
\begin{enumerate}
\item[\textit{(i)}] If $\alpha = 0$ or $k=0$, or if $\alpha = 1$ and $q_S =w^{\xi}_{p,1}, q_L=w^{-\xi}_{p,1} \oplus q_{1, \beta}$ for $\xi \in \left\{ \pm 1\right\}$, all primitive embeddings of $S$ in $L$ compatible with Proposition \ref{discriminant groups structure} define isomorphic sublattices; in particular, the orthogonal complement $T$ is uniquely determined by the invariants of $S$. 
\item[\textit{(ii)}] Otherwise, provided that the natural homomorphism $O(S) \ra O(q_S)$ is surjective, there are at most two distinct isomorphism classes for the orthogonal complement $T$ of the image of a compatible embedding $S \hookrightarrow L$, one with $l(A_T)=k+1$ and one with $l(A_T) = \max\{1,k-1\}$.
\end{enumerate}
\end{prop}

\begin{proof}
If $\alpha = 0$ or $k=0$, or if $\alpha = 1$ and $q_S =w^{\xi}_{p,1}, q_L=w^{-\xi}_{p,1} \oplus q_{1, \beta}$ for $\xi \in \left\{ \pm 1\right\}$, by Lemma \ref{embeddings of S} a (compatible) primitive embedding of $S$ in $L$ is characterized by a pair $(T, \gamma_T)$, with $T$ a lattice of signature $(1, 22 - (p-1)m)$, $q_T = (-q_S)\oplus q_L$ and $\gamma_T \in O(q_T)$. In this case, then, $l(A_T) = \max\{1, k + \alpha\}$, because $(p, \beta) = 1$. Therefore, if such an indefinite lattice $T$ exists and if $l(A_T) \leq \rk(T)-2$ (i.e.\ if $k \leq 21 - \alpha - (p-1)m$), by Theorem \ref{l<rk-2} $T$ is uniquely determined, up to isometries. This assumption also guarantees that the natural morphism $O(T) \ra O(q_T)$ is surjective (Theorem \ref{l<rk-2} again), therefore different choices of $\gamma_T$ give isomorphic primitive sublattices $S$ in $L$.

Assume now that we are not in one of the cases of point (i) (in particular, let $\alpha = 1$ and $k \geq 1$); moreover, suppose that $O(S) \ra O(q_S)$ is surjective and $k \leq 20 - (p-1)m$. A compatible embedding $i: S \hookrightarrow L$ is determined by a triple $(x, T, \gamma_T)$ as in Lemma \ref{embeddings of S}. We make a distinction:
\begin{itemize}
\item Triples $(0,T,\gamma_T)$ correspond to embeddings where $q_T = (-q_S)\oplus q_L$, so $l(A_T) = k + 1$. Then, as before, from the assumption $k \leq 20 - (p-1)m$ we get that all these embeddings define isomorphic sublattices in $L$ and that $T$ is uniquely determined.
\item If $x \neq 0$, the triple $(x, T, \gamma_T)$ was obtained, in the proof of Lemma \ref{embeddings of S}, from a quintuple $\Theta_i = (H_S, H_L, \gamma, T, \gamma_T)$, with $H_S = \langle x \rangle \subset A_S$, $H_L = \langle e \rangle \subset A_L$. If we now consider a different quintuple $\Theta_{i'}$, with $H'_S = \langle x' \rangle$ and $x' \neq 0$, the embeddings $i, i'$ will define isomorphic sublattices of $L$. This follows from Lemma \ref{embeddings of S} and Theorem \ref{nik embedd}, because -- under our assumptions -- two different subgroups $H_S, H'_S \subset A_S$ as above are conjugated by an automorphism of $S$. In fact, the restrictions $q_S\vert_{H_S}$ and $q_S\vert_{H'_S}$ are non-degenerate and isomorphic, since they both are $\frac{\IZ}{p \IZ} \left( -\frac{\beta}{p}\right)$; therefore, by Proposition \ref{nik 1.2.1} and the classification of $p$-elementary forms, also the forms on $H_S^\perp$ and $(H'_S)^\perp$ will coincide. This implies that there exists an automorphism of $A_S$ which exchanges $H_S$ and $H'_S$, and -- by the surjectivity of $O(S) \ra O(q_S)$ -- this automorphism is induced by an automorphism of $S$.

We conclude that the isometry class of $S$ as a primitive sublattice of $L$ does not depend on the choice of $x \neq 0$, nor of $T, \gamma_T$, since $k \leq 20 - (p-1)m$ and here $l(A_T) = \max\{1, k -1\}$, so $l(A_T) \leq \rk(T)-2$. \qedhere
\end{itemize} 
\end{proof}

Adopting the terminology used in \cite{bcs}, we provide the following definition.

\begin{defi}\label{admissible triple def}
Let $(p,m,a)$ be a triple of integers, with $3 \leq p \leq 23$ prime, $m \geq 1$, $(p-1)m \leq 22$ and $0 \leq a \leq \min\left\{m, 23-(p-1)m \right\}$. The triple is said to be \emph{admissible}, for a given integer $n \geq 2$, if there exist two orthogonal sublattices $T,S \subset L = U^{\oplus 3}\oplus E_8^{\oplus 2} \oplus \langle-2(n-1)\rangle$ such that: $\signt(T) = (1,22-(p-1)m)$, $\signt(S) = (2, (p-1)m-2)$, $\frac{L}{T \oplus S} \cong \left( \frac{\IZ}{p \IZ}\right)^{\oplus a}$ and the discriminant groups $A_T$ and $A_S$ are as in Proposition \ref{discriminant groups structure}.  
\end{defi}

\begin{rem} \label{remark a=0}
The condition $\frac{L}{T \oplus S} \cong \left( \frac{\IZ}{p \IZ}\right)^{\oplus a}$ implies that all admissible triples of the form $(p,m,0)$ will define orthogonal sublattices $T,S \subset L$ such that $L = T \oplus S$.
\end{rem}

We can now rephrase Proposition \ref{unicity embedding and Tbis} in the following way, taking into account the uniqueness of $S$ too.

\begin{theorem}\label{thm: admissible triples}
Let $2(n-1) = p^\alpha \beta$, with $(p, \beta) = 1$ and $\alpha \leq 1$. If $(p,m,a)$ is an admissible triple, there exists a unique even $p$-elementary lattice $S$ as in Definition \ref{admissible triple def}, up to isometries. Its primitive embedding in $L$ and its orthogonal complement $T \subset L$ are uniquely determined (up to isometries of $L$) by $(p,m,a)$, if $O(S)\ra O(q_S)$ is surjective and $l(A_T) \leq 21 - (p-1)m$.
\end{theorem}

\begin{proof}
By Proposition \ref{discriminant groups structure} and Remark \ref{parity} the discriminant group $A_S$ is $\left( \frac{\IZ}{p \IZ}\right)^{\oplus a}$ if $m$ and $a$ have same parity, otherwise $A_S \cong \left( \frac{\IZ}{p \IZ}\right)^{\oplus a+1}$. Moreover, since the triple $(p,m,a)$ fixes the signature of $S$, it also fixes the quadratic form on $A_S$, as we explained in Remark \ref{forme su S}. Thus, if $\rk(S) \geq 3$ the lattice $S$ is unique, up to isometries of $S$, by \cite[Theorem 2.2]{bcs}. The same is shown to hold also for the remaining cases, i.e.\ the triples $(3,1,0)$ and $(3,1,1)$, where $S$ is positive definite of rank two (see \cite[Table 15.1]{conway_sloane}). Then, Remark \ref{parity} and Proposition \ref{unicity embedding and Tbis} give the statement under the assumptions $O(S)\ra O(q_S)$ surjective and $l(A_T) \leq 21 - (p-1)m$.
\end{proof}

\begin{rem}
If both triples $(p,m,a)$, $(p,m,a+1)$ are admissible, with $m$ and $a$ of different parity, then they determine the same lattice $S$, up to isometries. As a matter of fact, in both cases the signature of $S$ is $(2, (p-1)m - 2)$ and, by Remark \ref{parity}, its discriminant group is $\left( \frac{\IZ}{p \IZ}\right)^{\oplus a+1}$. Notice, however, that the invariant lattices $T$ corresponding to the two triples are non-isometric.
\end{rem}

To conclude this subsection, we apply our results to explicitly list all possible quadratic forms $q_S, q_T$, up to isometries, on the discriminant groups $A_S, A_T$: by Lemma \ref{embeddings of S}, we will need to discuss separately the cases $\alpha = 0$ and $\alpha =1$ and to distinguish on whether $-\beta$ is a quadratic residue modulo $p$. This classification of quadratic forms is needed for listing admissible pairs of lattices $(T,S)$ for specific values of $n$ and $p$. 

\begin{prop}\label{prop: forms q_S and q_T}
Let $2(n-1) = p^\alpha \beta$, with $(p,\beta)=1$ and $\alpha \leq 1$. Let $(p,m,a)$ be an admissible triple and $T,S$ lattices corresponding to it; if $\alpha = 1$, also assume that the natural group homomorphism $O(S) \ra O(q_S)$ is surjective. Then one of the following holds:
\begin{enumerate}
\item[\textit{(i)}] $q_T=(-q_S) \oplus q_L$, with $q_S = \left( w^{+1}_{p,1} \right)^{\oplus a}$ or $q_S = \left( w^{+1}_{p,1} \right)^{\oplus a-1} \oplus w^{-1}_{p,1}$;
\item[\textit{(ii)}] $\alpha = 1$, $-\beta$ is a quadratic residue modulo $p$ and
\begin{enumerate}
\item $q_S = \left( w^{+1}_{p,1} \right)^{\oplus a+1}, q_T = \left( -w^{+1}_{p,1} \right)^{\oplus a} \oplus q_{1, \beta}$, or
\item $a \geq 1, q_S = \left( w^{+1}_{p,1} \right)^{\oplus a} \oplus w^{-1}_{p,1}, q_T = \left(- w^{+1}_{p,1} \right)^{\oplus a-1} \oplus (-w^{-1}_{p,1}) \oplus q_{1, \beta}$.
\end{enumerate}
\item[\textit{(iii)}] $\alpha = 1$, $-\beta$ is not a quadratic residue modulo $p$ and
\begin{enumerate}
\item $a \geq 1, q_S = \left( w^{+1}_{p,1} \right)^{\oplus a+1}, q_T = \left( -w^{+1}_{p,1} \right)^{\oplus a-1} \oplus (-w^{-1}_{p,1}) \oplus q_{1, \beta}$, or
\item $q_S = \left( w^{+1}_{p,1} \right)^{\oplus a} \oplus w^{-1}_{p,1}, q_T = \left( -w^{+1}_{p,1} \right)^{\oplus a} \oplus q_{1, \beta}$.
\end{enumerate}
\end{enumerate}
\end{prop}

\begin{proof}
As explained in the proof of Lemma \ref{embeddings of S}, case (i) corresponds to embeddings $S \hookrightarrow L$ determined by quintuples $(H_S, H_L, \gamma, T, \gamma_T)$ with $H_S= 0, H_L = 0$; moreover, the quadratic form $q_S$ is as in (\ref{form q_S}), with $k = l(A_S) = a$ by Proposition \ref{discriminant groups structure}. This is the only possibility when $\alpha = 0$; if instead $\alpha = 1$, there may also be compatible embeddings $S \hookrightarrow L$ corresponding to quintuples with $H_S \neq 0$ (see again Lemma \ref{embeddings of S}): in this case, by the surjectivity of $O(S) \ra O(q_S)$, as we showed in the proof of Proposition \ref{unicity embedding and Tbis} the subgroup $H_S$ can be regarded -- up to changing the generators of $A_S$ -- as one of the direct summands in the representation (\ref{form q_S}) of the quadratic form $q_S$. On the discriminant group of the orthogonal complement $T \subset L$, the quadratic form is then $q_T = ((-q_S)\oplus q_L)\rvert_{\Gamma^\perp/\Gamma}$, with $q_L$ as in (\ref{form q_L}) and $q_S$ as in (\ref{form q_S}), where now $k = l(A_S) = a+1$ by Proposition \ref{discriminant groups structure}.

Let's assume that $-\beta$ is a quadratic residue modulo $p$, so that $q_L = w^{+1}_{p,1} \oplus q_{1, \beta}$, and suppose $q_S = \left( w^{+1}_{p,1} \right)^{\oplus a+1}$. Adopting the same notations used in the previous proofs, let $x \in A_S$ be the generator of the subgroup corresponding to one of the summands $w^{+1}_{p,1}$ in $q_S$ and $e$ be the generator of $\IZ/p \IZ \subset A_L$. Then $H_S = \langle x\rangle$, $H_L = \langle e\rangle$, $\gamma: x \mapsto e$ and the graph of $\gamma$ is $\Gamma = \langle (x,e)\rangle \subset A_S \oplus A_L$. A direct computation shows that, with respect to the quadratic form $(-q_S) \oplus q_L$ on $A_S \oplus A_L$, the orthogonal of $\Gamma$ is
\[ \Gamma^\perp = \left( H_S^\perp \oplus H_L^\perp \right) + \Gamma.\]

This implies that the quadratic form $q_T \cong ((-q_S)\oplus q_L)\rvert_{\Gamma^\perp/\Gamma}$ is isometric to the restriction of $(-q_S)\oplus q_L$ to $H_S^\perp \oplus H_L^\perp$, therefore $q_T = \left( -w^{+1}_{p,1} \right)^{\oplus a} \oplus q_{1, \beta}$.

If instead $q_S = \left( w^{+1}_{p,1} \right)^{\oplus a} \oplus w^{-1}_{p,1}$, we need to ask $a \geq 1$, otherwise it is not possible to find subgroups $H_S \subset A_S$ and $H_L \subset A_L$ such that $q_{S}\rvert_{H_S} \cong q_{L}\rvert_{H_L}$. As in the previous case, we can assume $H_S = \langle x\rangle$, $H_L = \langle e\rangle$, $\gamma: x \mapsto e$ with $x \in A_S$ generator of one of the components $w^{+1}_{p,1}$ in $q_S$ and $e \in A_L$ generator of the summand $w^{+1}_{p,1}$ of $q_L$. Since $\Gamma$, $\Gamma^\perp$ are the same as above, the form $q_T$ still arises as the restriction of $(-q_S)\oplus q_L$ to $H_S^\perp \oplus H_L^\perp$, and therefore $q_T = \left(- w^{+1}_{p,1} \right)^{\oplus a-1} \oplus (-w^{-1}_{p,1}) \oplus q_{1, \beta}$.

The two cases where $q_L = w^{-1}_{p,1} \oplus q_{1, \beta}$ (i.e.\ $-\beta$ is not a quadratic residue modulo $p$) can be discussed in an analogous way.
\end{proof}

\subsection{A special case: $\rk(T)=1$}\label{subsection: rkT=1}
In this subsection we focus on the cases where $T$ has rank one, which correspond to maximal dimensional families of manifolds of $\hskn$-type equipped with a non-symplectic automorphism. Since $T$ has rank one, $\rk(S) = (p-1)m = 22$: for $p$ odd, this can only happen if $p = 3, m = 11$ or $p = 23, m = 1$. As before, we write $2(n-1) = p^\alpha \beta$, with $(p, \beta) = 1$.

If $\alpha = 0$, then $a$ must be odd, because it needs to be of the same parity as $m$ (Remark \ref{parity}); in particular, $a \geq 1$. Moreover $A_T \cong \left(\frac{\IZ}{p \IZ}\right)^{\oplus a} \oplus \frac{\IZ}{p^\alpha \IZ} \oplus \frac{\IZ}{\beta \IZ}$, by Proposition \ref{discriminant groups structure}; since $\rk(T)=1$, then necessarily $\alpha = 0$, $a = 1$. We conclude $T = \langle 2p(n-1) \rangle$ ($\alpha=0$ means that $p$ and $2(n-1)$ are coprime).

If instead $\alpha \geq 1$, there are two possibilities:
\begin{itemize}
\item $a$ odd. Then $A_T \cong \left(\frac{\IZ}{p \IZ}\right)^{\oplus a} \oplus \frac{\IZ}{p^\alpha \IZ} \oplus \frac{\IZ}{\beta \IZ}$ with $\alpha \geq 1$ and $a \geq 1$. As a consequence $l(A_T) \geq 2$, so $T$ cannot have rank one.
\item $a$ even. By the classification provided in Proposition \ref{discriminant groups structure}, $T$ cannot be of rank one if $a > 0$. Hence $\alpha = 1$, $a=0$, $T = \langle \beta \rangle = \langle \frac{2(n-1)}{p} \rangle$.
\end{itemize}

Moreover, we need to impose conditions on the orthogonal lattice $S$, using again Proposition \ref{discriminant groups structure}. Since $\rk(T) = 1$, we can also use \cite[Proposition 3.6]{ghs} to determine the existence and the structure of such primitive sublattices $T, S \subset L$. We do it separately for the two possible cases we found.

\begin{itemize}
\item $\alpha = 0, a = 1, T= \langle h\rangle$, with $h \in L$ primitive vector of length $h^2=2p(n-1)$.\\ By \cite[Proposition 3.6]{ghs}, the orthogonal lattice $S$ has discriminant $\frac{4p(n-1)^2}{f^2}$, where $f$ is the generator of the ideal $(h,L) \subset \IZ$. By Proposition \ref{discriminant groups structure} we know that $A_S \cong \frac{\IZ}{p \IZ}$, therefore $\discr (S) = p$ and we need $f = 2(n-1)$. Applying again \cite[Proposition 3.6]{ghs}, we can conclude that such a $T$ exists if and only if $-p$ is a quadratic residue modulo $4(n-1)$.
\item $\alpha = 1, a = 0, T= \langle h\rangle$, with $h \in L$ primitive vector of length $h^2=\frac{2(n-1)}{p}$.\\ We have $A_S \cong \frac{\IZ}{p \IZ}$, by Proposition \ref{discriminant groups structure}, and $p = \discr (S) = \frac{4(n-1)^2}{pf^2}$, so $f=\frac{2(n-1)}{p}$. Here $p^2 \nmid 4(n-1)$, so $p$ is invertible modulo $\frac{4(n-1)}{p}$, hence by \cite[Proposition 3.6]{ghs} such a $T$ exists if and only if $-p$ is a quadratic residue modulo $\frac{4(n-1)}{p}$.
\end{itemize}

We rephrase these results as follows.

\begin{prop}\label{rkT=1bis}
Let $p \geq 3$ be a prime and $2(n-1) = p^\alpha \beta$ with $(p, \beta) = 1$. A triple $(p,m,a)$, with $(p-1)m = 22$, is admissible if and only if $\alpha \in \left\{ 0, 1 \right\}$, $a = 1-\alpha$ and $-p$ is a quadratic residue modulo $\frac{4(n-1)}{p^\alpha}$.

If this happens, then one of the following holds:
\begin{enumerate}
\item $\alpha = 0, p=3, m=11, a=1, T=\langle 6(n-1)\rangle, S=U^{\oplus 2}\oplus E_8^{\oplus 2}\oplus A_2$;
\item $\alpha = 1, p=3, m=11, a=0, T= \langle \beta \rangle= \langle \frac{2(n-1)}{3}\rangle, S=U^{\oplus 2}\oplus E_8^{\oplus 2}\oplus A_2$;
\item $\alpha = 0, p=23, m=1, a=1, T =\langle 46(n-1)\rangle, S=U^{\oplus 2}\oplus E_8^{\oplus 2}\oplus K_{23}$;
\item $\alpha = 1, p=23, m=1, a=0, T = \langle \beta \rangle= \langle \frac{2(n-1)}{23}\rangle, S=U^{\oplus 2}\oplus E_8^{\oplus 2}\oplus K_{23}$.
\end{enumerate}
\end{prop}

\begin{proof}
The explicit description of the lattice $S$ in the four cases is obtained by combining \cite[Proposition 3.6]{ghs} (where $S$ is represented as $S=U^{\oplus 2}\oplus E_8^{\oplus 2}\oplus B$ for a negative definite, even lattice $B$ of rank $2$ depending on $p,n,f$) with the results on lattice isomorphisms given in \cite[Corollary 1.13.5]{nikulin} and \cite[\S 1]{rudakov-shafarevich}, which guarantee the uniqueness, up to isometries, of $p$-elementary lattices of signature $(2,20)$ and length one, when $p = 3$ or $p = 23$.
\end{proof}

From Proposition \ref{rkT=1bis} it follows, for instance, that the triple $(3,11,1)$ is admissible for $n=2$, as already observed in \cite{bcs}, because $-3$ is a quadratic residue modulo $4$: in this case, we have $T = \langle 6\rangle$. Similarly, $(3,11,0)$ is admissible when $n=4$ (here $\alpha=1$ and $-3$, again, is a quadratic residue modulo $4$), with $T = \langle 2\rangle$. Instead, $(3,11,1)$ is not admissible when $n=3$, because $-3$ is not a quadratic residue modulo $8$.

The triple $(23,1,1)$ was already found to be admissible for $n=2$ in \cite{bcms_p=23} (where the authors also gave the isomorphism classes of $T,S$). By our proposition, this triple is admissible for $n=3,4$ too, since $-23 \equiv 
1$ both modulo $8$ and modulo $12$. Finally, the first value of $n$ such that $(23,1,0)$ is admissible is $n = 24$, since $2(n-1) = 46 = 23 \cdot 2$ and $-23$ is a quadratic residue modulo $4$.

\subsection{Admissible triples for $n=3,4$} \label{subsect: classif for n=3,4}
In this section we provide a complete classification of admissible triples $(p,m,a)$ for $n=3,4$. In both cases, for any odd prime number $p$ we have $\alpha \leq 1$, therefore Theorem \ref{thm: admissible triples} allows us to exhibit the lattices $T, S$ (up to isometries) for each triple. This classification of the two lattices is achieved by direct computation for all possible triples $(p,m,a)$, checking for each of them if lattices $T, S$ as in Definition \ref{admissible triple def} exist or not. To do so, we apply Theorem \cite[Theorem 1.10.1]{nikulin}, which provides necessary and sufficient conditions for the existence of an even lattice with given signature and discriminant form.

\textbf{Manifolds of $K3^{[3]}$-type}.
\begin{itemize}
\item For $p=23$ there is only one admissible triple, namely $(23,1,1)$, as we already observed in \S \ref{subsection: rkT=1}: the isometry classes of $S$ and $T$ are given in Proposition \ref{rkT=1bis}, case $(3)$.

\item For all primes $5 \leq p \leq 19$, the admissible triples and the lattices $S$ are the ones listed for $n=2$ in the tables of \cite[Appendix A]{bcs}, while the lattices $T$ can be obtained from the corresponding ones in the tables by switching $\langle-2\rangle$ with $\langle-4\rangle$ in their description, since now $L = U^{\oplus 3}\oplus E_8^{\oplus 2} \oplus \langle-4\rangle$. Notice that, with respect to \cite[Table 5]{bcs}, by Remark \ref{parity} we can now say that the triple $(13,1,0)$ is not admissible, neither for $n=2$ nor for $n=3$: in fact, for these values of $n$ we have $\alpha = 0$ for all possible primes $p$, hence $m$ and $a$ always need to have the same parity.

\textbf{Example: $(p,m,a) = (5,5,3)$}. This triple is not admissible for $n=2$ and it is checked to be still not admissible for $n=3$. In fact, for these values of $(p,m,a)$ the lattice $S$ would be isomorphic to $U(5) \oplus E_8^{\oplus 2} \oplus H_5$, whose discriminant group is $A_S \cong A_{U(5)} \oplus A_{H_5} \cong \frac{\IZ}{5 \IZ} \left( \frac{2}{5}\right)^{\oplus 3}$. As we pointed out at the beginning of \S \ref{subsection: adm triples}, $A_L \cong \frac{\IZ}{4 \IZ}\left( -\frac{1}{4}\right)$, therefore if $S$ admitted an embedding in $L$ the quadratic form on $T$ would be \mbox{$q_T = \frac{\IZ}{5 \IZ}\left( \frac{8}{5}\right)^{\oplus 3} \oplus \frac{\IZ}{4 \IZ} \left( -\frac{1}{4}\right)$} by Lemma \ref{embeddings of S} and $\sign(T) = (1,2)$. By \cite[Theorem 1.10.1]{nikulin}, a lattice $T$ with these invariants exists only if its $5$-adic completion $T_5 \coloneqq T \otimes_\IZ \IZ_5$ is such that $\lvert A_T \rvert \equiv \discr(K) \text{ mod} \left( \IZ_5^*\right)^2$, where $K$ is the unique $5$-adic lattice of rank $l(A_{T_5})$ and discriminant form $q_{T_5}$ (see \cite[Theorem 1.9.1]{nikulin}). In our case, since $A_{T_5} \cong \left( A_T\right)_5 \cong \frac{\IZ}{5 \IZ}\left( \frac{8}{5}\right)^{\oplus 3}$, using \cite[Proposition 1.8.1]{nikulin} we compute $K = \langle 5 \cdot \frac{1}{8}\rangle ^{\oplus 3}$, where $\frac{1}{8} \in \IZ_5^*$. Thus, $\lvert A_T \rvert = 4 \cdot 5^3$ and $\discr(K) = \left( \frac{5}{8}\right)^3$: these two values do not satisfy the relation $\lvert A_T \rvert  \equiv \discr(K) \text{ mod} \left( \IZ_5^*\right)^2$, because $2^{11} \notin \left( \IZ_5^*\right)^2$ (it follows from the fact that $3$ is not a quadratic residue modulo $5$). We conclude that a lattice $T$ with such signature and quadratic form does not exist.

\item For $p=3$, Table \ref{n=3,ord3} in Appendix \ref{appendix: tables} lists all admissible triples, with the corresponding isomorphism classes for $T,S$: as for higher primes, we can find many similarities with the analogous table for $n=2$ (\cite[Table 1]{bcs}). However, there are also some significant differences.
\begin{itemize}
\item As we observed in \S \ref{subsection: rkT=1}, there are no admissible triples with $m=11$.
\item The triple $(3,9,5)$ is now admissible: here $S= U(3)^{\oplus 2} \oplus E_6 \oplus E_8$, while $\sign(T)=(1,4)$ and $q_T = -q_S \oplus q_L \cong \frac{\IZ}{3 \IZ}\left( \frac{4}{3}\right)^{\oplus 5} \oplus  \frac{\IZ}{4 \IZ} \left( -\frac{1}{4}\right)$. The existence of a lattice $T$ with these invariants is proved applying \cite[Theorem 1.10.1]{nikulin} and there is a unique isometry class in the genus of $T$ by \cite[Chapter 15, Theorem 21]{conway_sloane}. In particular, we can take $T = U(3) \oplus \Omega$, where $\Omega$ is the even lattice of rank three whose bilinear form is defined by the matrix
\[ \Omega \coloneqq \begin{pmatrix}
-6 & 0 & -3\\
0 & -6 & 9\\
-3 & 9 & -18
\end{pmatrix}.\]
We have $\sign(\Omega) = (0,3)$ and $q_\Omega = \frac{\IZ}{3 \IZ}\left( \frac{4}{3}\right)^{\oplus 2} \oplus \frac{\IZ}{3 \IZ}\left( \frac{2}{3}\right) \oplus  \frac{\IZ}{4 \IZ} \left( -\frac{1}{4}\right)$, therefore $q_{U(3) \oplus \Omega} \cong -q_S \oplus q_L$ (using \cite[Proposition 1.8.2]{nikulin}).
\item An additional new admissible triple is $(p,m,a)=(3,8,6)$: here we compute $S = U(3)^{\oplus 2} \oplus E_6^{\oplus 2}$, therefore $\sign(T)=(t_{(+)},t_{(-)})=(1,6)$ and $q_T = \frac{\IZ}{3 \IZ}\left( \frac{4}{3}\right)^{\oplus 6} \oplus  \frac{\IZ}{4 \IZ} \left( -\frac{1}{4}\right)$. In this case, the strict inequality \mbox{$t_{(+)} + t_{(-)} > l(A_T)$} holds: since moreover $t_{(+)} - t_{(-)} \equiv \sign(q_T)$ (mod $8$), such a lattice $T$ exists by \cite[Corollary 1.10.2]{nikulin} and again it is unique (up to isometries) by \cite[Chapter 15, Theorem 21]{conway_sloane}. A representative of this genus is $T = U(3) \oplus A_2 \oplus \Omega$. 
\end{itemize}
\end{itemize}

\textbf{Manifolds of $\hskq$-type}.
\begin{itemize}
\item For $p=23$ we have that $(23,1,1)$ is the only admissible triple (see \S \ref{subsection: rkT=1}): the isomorphism classes of $T,S$ were obtained in Proposition \ref{rkT=1bis}.

\item For primes $5 \leq p \leq 19$, again the lattices $T,S$ and all admissible triples are the ones listed in the tables of \cite[Appendix A]{bcs} (apart from $(13,1,0)$, which is not admissible), up to replacing the $\langle -2 \rangle$ summand with a $\langle -6 \rangle$ summand in $T$.

\item The last prime we need to consider is $p=3$. This is the first case we encounter where an odd $p$ divides $2(n-1)$: in particular, $2(n-1) = 6 = 3^\alpha \beta$ with $\alpha = 1$ and $\beta = 2$. Since we have $\alpha = 1$, by Lemma \ref{embeddings of S} and Proposition \ref{unicity embedding and Tbis} we know that we can expect to have many more admissible triples than the ones we found for $p=3$ and $n=2,3$: in fact, the same lattice $S$ might be embedded in $L$ in two non-isomorphic ways. Table \ref{n=4,ord3} (Appendix \ref{appendix: tables}) contains the list of all admissible triples and of the corresponding isomorphism classes for the lattices $T,S$. In particular, the triple $(3,11,0)$ is admissible thanks to Proposition \ref{rkT=1bis}; some other triples, such as $(3,8,6)$ and $(3,8,7)$, are excluded again by use of \cite[Theorem 1.10.1]{nikulin}.
\end{itemize}

\section{Existence of automorphisms}

The classification of admissible lattices $T,S$ presented in Section \ref{section: isometries} does not tell us which cases can be realized by actual automorphisms. In this section we provide several tools to construct non-symplectic automorphisms of odd prime order on manifolds of $\hskn$-type, which are valid for any $n \geq 2$. In particular, we are interested in two types of manifolds: Hilbert schemes of points on $K3$ surfaces and moduli spaces of (possibly twisted) sheaves on $K3$'s. Moreover, in \S \ref{subsection: existence for rkT=1} we show that the existence of automorphisms which realize admissible pairs $(T,S)$ where $T$ has rank one can always be proved using the global Torelli theorem for IHS manifolds.

\subsection{Natural automorphisms} \label{subsect: natural auto}
 
Let $\Sigma$ be a smooth $K3$ surface. An automorphism $\varphi \in \aut(\Sigma)$ induces an automorphism $\varphi^{[n]}$ on the Hilbert scheme $\Sigma^{[n]}$, by setting $\varphi^{[n]}(\xi) = \varphi(\xi)$ for any zero-dimensional subscheme $\xi \subset \Sigma$ of length $n$ (i.e.\ a point in $\Sigma^{[n]}$). Such an automorphism $\varphi^{[n]}$ is said to be \emph{natural}.

By \cite[Proposition 6]{beauville}, we have an injection $i: H^2(\Sigma, \IC) \hookrightarrow H^2(\Sigma^{[n]}, \IC)$ compatible with the Hodge structures and such that
\[ H^2(\Sigma^{[n]}, \IC) = i\left( H^2(\Sigma, \IC) \right) \oplus \IC[E]\]
\noindent where $[E]$ is the class of the exceptional divisor of the Hilbert-Chow morphism. In particular, if $\varphi \in \aut(\Sigma)$ is a non-symplectic automorphism, then $\varphi^{[n]}$ will also act non-symplectically on $\Sigma^{[n]}$. In fact, if $\omega \in H^{2,0}(\Sigma)$ then $i(\omega) \in H^{2,0}(\Sigma^{[n]})$ and $\left( \varphi^{[n]} \right)^*(i(\omega)) = i\left( \varphi^*(\omega)\right)$ by \cite[Theorem 1]{bs}.

Moreover, $H^2(\Sigma^{[n]}, \IZ) = i\left( H^2(\Sigma, \IZ) \right) \oplus \IZ \delta$, where $\delta \in H^2(\Sigma^{[n]}, \IZ)$ is a class such that $2\delta = [E]$ and $H^2(\Sigma, \IZ) \cong L_{K3} \coloneqq U^{\oplus 3} \oplus E_8^{\oplus 2}$. As observed in \cite[\S 3]{bs}, the action of the natural automorphism $\varphi^{[n]}$ on $H^2(\Sigma^{[n]}, \IZ)$ can be decomposed as $\left( \varphi^{[n]} \right)^* = (\varphi^*, \id_{\IZ \delta})$. This implies that $T_{\varphi^{[n]}} = T_\varphi \oplus \IZ \delta \cong T_\varphi \oplus \langle -2(n-1) \rangle$, while $S_{\varphi^{[n]}} = S_\varphi$, with $(T_\varphi, S_\varphi) \subset L_{K3}$.

We conclude that all admissible triples $(p,m,a)$ where $T \cong T_{K3} \oplus \langle -2(n-1) \rangle$ and $S \cong S_{K3}$, with $(T_{K3}, S_{K3})$ the invariant lattice and its orthogonal complement for the action of a non-symplectic automorphism on a $K3$ surface, are realized by natural automorphisms. All possible isomorphism classes for the pairs $(T_{K3}, S_{K3})$  can be found in \cite{autom_k3_ord3} (order $p=3$) and \cite{autom_k3} (prime order $5 \leq p \leq 19$), therefore it is immediate to check -- for any $n$ -- which admissible cases admit a natural realization.

In the tables of Appendix \ref{appendix: tables} we mark with the symbol $\clubsuit$ the triples realized by natural automorphisms. For $n=4$ (Table \ref{n=4,ord3}), it may not always be immediate to recognize the lattices $T_{K3}$ of \cite[Table 2]{autom_k3_ord3} as direct summands in the lattices $T$ we provide, since we often choose different representatives in the same isomorphism classes. In particular, we have the following isometries: $U \oplus E_6 \oplus A_2 \cong U(3) \oplus E_8$; $U \oplus A_2^{\oplus 3} \cong U(3) \oplus E_6; U(3) \oplus A_2^{\oplus 3} \cong U \oplus E_6^\vee(3)$ (they can all be proved using Theorem \ref{l<rk-2}).  The reason why we adopt different genus representatives for these lattices will become clear in \S \ref{subsection: induced autom n=4} (Lemma \ref{lemma: induced n=4}).  

\subsection{Induced automorphisms} \label{subsect: induced auto}

A direct generalization of the notion of natural automorphisms is given by \emph{induced automorphisms}, which were first introduced and studied in \cite{ow}, \cite{mw} and later extended to the case of twisted $K3$ surfaces in \cite{ckkm}.

We recall here the fundamental definitions and results (see \cite[\S 2.3, \S 3]{ckkm} for additional details and references). Let $\left( \Sigma, \alpha \right)$ be a twisted $K3$ surface, where \mbox{$\alpha \in \br(\Sigma) \coloneqq H^2\left(\Sigma, \mathcal{O}^*_{\Sigma}\right)_{\text{tor}}$} is a Brauer class. By \cite[\S 2]{vg_brauer}, if $\alpha$ has order $k$ then it can be identified with a surjective homomorphism $\alpha: \trans(\Sigma) \ra \IZ/k \IZ$, where $\trans(\Sigma) \subset L_{K3}$ is the transcendental lattice of the surface. Using the exponential sequence we can find a $B$-field lift of $\alpha$, i.e.\ a class $B \in H^2(\Sigma, \IQ)$ such that $kB$ is integral and the map $\alpha: \trans(\Sigma) \ra \IZ/k \IZ$ is just the intersection product with $kB$ (see \cite[\S 3]{stellari_huybrechts}). In particular, $B$ is defined up to an element in $H^2(\Sigma, \IZ) + \frac{1}{k}\pic(\Sigma)$.

The cohomology ring $H^*(\Sigma, \IZ) = H^0(\Sigma, \IZ) \oplus H^2(\Sigma, \IZ) \oplus H^4(\Sigma, \IZ)$ admits a lattice structure, with pairing $(r,H,s)\cdot (r',H',s') = H\cdot H' - rs' -r's$. As a lattice, $H^*(\Sigma, \IZ)$ is isometric to the Mukai lattice $\Lambda_{24} = U^{\oplus 4} \oplus E_8^{\oplus 2}$. Moreover, a Mukai vector $v = (r,H,s)$ is \emph{positive} if $H \in \pic(\Sigma)$ and either $r > 0$, or $r = 0$ and $H \neq 0$ effective, or $r = H = 0$ and $s > 0$. Starting from a primitive, positive vector $v=(r,H,s) \in H^*(\Sigma, \IZ)$ and a $B$-field lift $B$ of $\alpha$ we can define the twisted Mukai vector $v_B \coloneqq (r, H + rB, s + B \cdot H + r\frac{B^2}{2})$. Then, the coarse moduli space $M_{v}(\Sigma, \alpha)$ of $\alpha$-twisted sheaves with Mukai vector $v_B$ is an IHS manifold of $\hskn$-type, with $n = \frac{v_B^2}{2} + 1$, and such that $H^2(M_{v}(\Sigma, \alpha), \IZ) \cong v_{B}^\perp$ in the Mukai lattice (see \cite{baymacr} and \cite{yoshioka_twisted}).

Now, let $\varphi$ be an automorphism of $\Sigma$: a Brauer class $\alpha$ is invariant with respect to $\varphi$ if and only if $\alpha \circ \varphi^*\vert_{\trans(\Sigma)} = \alpha$. The following result holds (see \cite[Proposition 1.32]{mw} and \cite[\S 3]{ckkm}).

\begin{prop} \label{induced autom}
Let $(\Sigma, \alpha)$ be a twisted $K3$ surface, $\varphi$ an automorphism of $\Sigma$, $v$ a positive, primitive Mukai vector and $B$ a $B$-field lift of $\alpha$. If $v_B$ and $\alpha$ are $\varphi$-invariant, then $\varphi$ induces (via pullback of sheaves) an automorphism $\hat{\varphi}$ of $M_{v}(\Sigma, \alpha)$.
\end{prop} 

The automorphisms $\hat{\varphi}$ arising in this way are called \textit{twisted induced} (or just \textit{induced} in the non-twisted case, i.e.\ if $\alpha = 0$). 

\begin{prop} \cite[Theorem 3.4]{ckkm}
Let $\sigma$ be an automorphism of finite order on a manifold $X$ of $\hskn$-type acting trivially on $A_L$. Then $\sigma$ admits a realization as a twisted induced automorphism on a suitable moduli space $M_{v_B}(\Sigma, \alpha)$ if and only if the invariant lattice of the extension of $\sigma$ to the Mukai lattice contains primitively a copy of $U(d)$, for some multiple $d$ of the order of the Brauer class $\alpha$. 
\end{prop}

If $\alpha$ admits a $B$-field lift $B \in H^2(\Sigma, \IQ)$ such that $B^2 = B \cdot H = 0$, then the transcendental lattice of the moduli space $M_{v}(\Sigma, \alpha)$ is isomorphic to \mbox{$\ker(\alpha) \subset \trans(\Sigma)$}, which (if $\alpha \neq 0$) is a sublattice of index equal to the order of $\alpha$. By \cite[\S 3]{yoshioka_twisted}, $\pic(M_{v}(\Sigma, \alpha)) \cong v_B^\perp \cap \pic(\Sigma, \alpha)$ inside $H^*(\Sigma, \IZ)$, where $\pic(\Sigma, \alpha) \cong \pic(\Sigma) \oplus U$ if $\alpha = 0$, otherwise -- assuming the order of $\alpha$ is $k$ -- $\pic(\Sigma, \alpha)$ is generated by $\pic(\Sigma)$ and the vectors $(0,0,1), (k, kB, 0)$ by \cite[Lemma 3.1]{macrì_stellari}. 
\smallskip

In the case $\alpha = 0$, which was already studied in \cite{baymacr} and \cite{mw}, it is possible to provide some additional details on the action of induced automorphisms. Let $v \in H^*(\Sigma, \IZ)$ be a primitive, positive Mukai vector; then $M_{v}(\Sigma, 0)$ is isomorphic to the moduli space $M_\tau(v)$ of $\tau$-stable objects of Mukai vector $v$, for $\tau \in \stabil(\Sigma)$ a $v$-generic Bridgeland stability condition on the derived category $D^b(\Sigma)$ (see \cite{bridg} for details).

By our previous discussion, the transcendental lattice of $M_{\tau}(v)$ coincides with $\trans(\Sigma)$, while its Picard lattice is isomorphic to $v^\perp \cap \left( \pic(\Sigma) \oplus U \right)$. In particular, the summand $U$ in $\pic(\Sigma) \oplus U$ is just $H^0(\Sigma, \IZ) \oplus H^4(\Sigma, \IZ)$, which is the orthogonal complement of $H^2(\Sigma, \IZ) \cong L_{K3}$ inside $H^*(\Sigma, \IZ) \cong \Lambda_{24}$. Since $L_{K3}$ is unimodular, the action of an automorphism $\varphi \in \aut(\Sigma)$ on $L_{K3}$ extends to an action on $\Lambda_{24}$ which is trivial on $(L_{K3})^\perp$ (by \cite[Lemma 1.4]{mw}). Let $T_{K3},S_{K3} \subset L_{K3}$ and \mbox{$\hat{T}, \hat{S} \subset \Lambda_{24}$} be the invariant and co-invariant lattices of these two actions: by what we stated, $\hat{T} = T_{K3} \oplus U$ and $\hat{S} = S_{K3}$. The induced automorphism $\hat{\varphi}$ acts on $H^2(M_\tau(v), \IZ) \cong L = U^{\oplus 3} \oplus E_8^{\oplus 2} \oplus \langle -2(n-1)\rangle $: its invariant lattice is $T \cong \left( v^\perp \right)^\varphi = \hat{T} \cap v^\perp$ (see \cite[Lemma 1.34]{mw}). We rephrase the results of \cite[\S 2,3]{mw} as follows.

\begin{prop} \label{prop: induced}
Let $(p,m,a)$ be an admissible triple for $n$, with $(T,S)$ the corresponding pair of lattices; consider the canonical primitive embeddings $S \hookrightarrow L \hookrightarrow \Lambda_{24}$ and define $\hat{T} \coloneqq S^\perp \subset \Lambda_{24}$. Then the triple $(p,m,a)$ is realized by an induced automorphism if $\hat{T} \cong U \oplus T_{K3}$, $S \cong S_{K3}$, with $(T_{K3}, S_{K3})$ the invariant lattice and its orthogonal complement for the action of a non-symplectic automorphism on a $K3$ surface, and there exists a primitive vector $v \in \hat{T}$ of square $2(n-1)$ such that $T \cong v^\perp \cap \hat{T}$.
\end{prop}

In particular, all natural automorphisms can be considered as induced, since $\langle -2(n-1) \rangle$ is the orthogonal in $U$ of an element of square $2(n-1)$ (see \cite{mw}).

\smallskip
In \S \ref{section: induced for n=3,4} we will apply the theory of induced (and twisted induced) automorphisms to construct geometric realizations of several admissible triples for manifolds of type $\hskt$ and $\hskq$.

\subsection{Existence for $\rk(T) = 1$} \label{subsection: existence for rkT=1}

The global Torelli theorem (Theorem \ref{torelli}) can be applied to prove the existence of automorphisms of manifolds of $\hskn$-type realizing the pairs of lattices $(T,S)$ classified in Proposition \ref{rkT=1bis}, i.e.\ for $\rk(T) = 1$.

\begin{prop}\label{prop: existence for rkT=1}
Let $(p,m,a)$ be an admissible triple as in Proposition \ref{rkT=1bis}, for a certain $n$, and let $T,S$ be the lattices associated to it; then, there exists a manifold $X$ of $\hskn$-type and a non-symplectic automorphism $f \in \aut(X)$ of order $p$ such that $T_f \cong T$ and $S_f \cong S$. 
\end{prop}

\begin{proof}
We discuss separately the four possible cases classified in Proposition \ref{rkT=1bis}, keeping the same numbering.

\emph{Case (2): $(3,11,0)$}. Here we have $2(n-1)=3\beta$, with $(3, \beta) =1$; the invariant and co-invariant lattices are $T=\langle \beta \rangle$ and $S = U^{\oplus 2} \oplus E_8^{\oplus 2} \oplus A_2$, which by Proposition \ref{rkT=1bis} can be seen as orthogonal sublattices of $L= U^{\oplus 3} \oplus E_8^{\oplus 2} \oplus \langle -2(n-1) \rangle$. We first construct a monodromy of the lattice $L$ having invariant lattice $T$ and co-invariant lattice $S$. The triple has $a=0$, therefore $L = T \oplus S$ (see Remark \ref{remark a=0}): an isometry $\phi \in O(L)$ can then be represented as $\phi = \gamma \oplus \psi$, with $\gamma \in O(T)$ and $\psi \in O(S)$. Moreover, since we want $\phi$ to be of order $3$ with invariant lattice $T$, we will need $\gamma = \id_T$ and $\psi$ of order $3$ with no non-zero fixed points.
 
By \cite[Theorem 3.3]{autom_k3_ord3}, there exist a $K3$ surface $\Sigma$ and a non-symplectic automorphism $\varphi \in \aut(\Sigma)$ of order $3$ with invariant lattice $T_{K3} = U$ and co-invariant $S_{K3} = U^{\oplus 2} \oplus E_8^{\oplus 2}$. Thus, the natural automorphism $\varphi^{[n]}$ on the Hilbert scheme $\Sigma^{[n]}$ will have invariant lattice $T' = U \oplus \langle -2(n-1) \rangle$ and co-invariant lattice $S' = S_{K3} = U^{\oplus 2} \oplus E_8^{\oplus 2}$; in particular, $T' \oplus S' = L$, meaning that the triple $(3,10,0)$ is  realized by a natural automorphism for all $n \geq 2$. Moreover, since $\varphi^{[n]}$ has odd order, it induces a monodromy of $L$ which acts as $+\id$ on the discriminant $A_L \cong A_{T'} \oplus A_{S'}$ (Theorem \ref{monodromy group}); the restriction of this monodromy to $S'$ is therefore an isometry $\mu \in O(S')$ of order $3$, with no non-zero fixed vectors, such that $\bar{\mu} = \id_{A_{S'}}$ and $\spinn^{S'}(\mu) = 1$. On our original lattice $S = S' \oplus A_2$ we now consider the isometry $\psi = \mu \oplus \rho_0$, where $\rho_0$ acts on $A_2 = \left( \IZ e_1 \oplus \IZ e_2, \left( \begin{smallmatrix} -2&1\\ 1&-2 \end{smallmatrix} \right) \right)$ as
 \[ \rho_0(e_1) = e_2, \qquad \rho_0(e_2) = -e_1 -e_2.\]
 
It is easy to check that $\rho_0$ is an isometry of order $3$ without non-zero fixed points, inducing the identity on the discriminant group $A_{A_2}$ (this isometry was also used in \cite[\S 6.6]{dolg_vangeem}). Notice that, since $A_2$ is negative definite, $\spinn^{A_2}(\rho_0) = 1$. We can then conclude that $\psi = \mu \oplus \rho_0$ is an isometry of $S$ of order $3$ with no non-zero fixed points and inducing the identity on the discriminant; moreover, since $\psi$ is defined as an orthogonal sum, $\spinn^S(\psi) = \spinn^{S'}(\mu) \cdot \spinn^{A_2}(\rho_0) = 1$ (the reflections appearing in the factorisation of $\psi_\mathbb{R}$ are the extensions to $S_\mathbb{R}$ of the ones which factorise $\mu_\mathbb{R}$ and $\left(\rho_0\right)_\mathbb{R}$ as transformations of the orthogonal subspaces $S'_\mathbb{R}, \left(A_2\right)_\mathbb{R} \subset S_\mathbb{R}$). By the same reasoning, $\spinn^{L}(\phi) = \spinn^{L}(\id_T \oplus \psi) = \spinn^{S}(\psi) = 1$; thus, $\phi$ is a monodromy operator, thanks to Theorem \ref{monodromy group}, with invariant lattice $T$ and co-invariant $S$. By generalizing \cite[Proposition 5.3]{bcs}, there exists a manifold $X$ of $\hskn$-type and a marking $\eta: H^2(X,\IZ) \ra L$ such that $\eta(\ns(X)) = T$. The monodromy $\phi$ is an Hodge isometry, since it preserves $H^{2,0}(X) = \mathbb{C}\omega_{X}$ (because $\ns(X) = \omega_X^{\perp} \cap H^2(X, \IZ)$). Moreover, since $\rk(T) = 1$, $\phi$ fixes a K\"ahler class (the generator of $\eta(\ns(X)) = T$); the global Torelli theorem (Theorem \ref{torelli}) allows us to conclude that there exists an automorphism $f \in \aut(X)$ such that $\eta \circ f^* \circ \eta^{-1}= \phi$.

\emph{Case (1): $(3,11,1)$}. In this case, $T=\langle 6(n-1) \rangle$ and $S = U^{\oplus 2} \oplus E_8^{\oplus 2} \oplus A_2$. Now $T \oplus S$ is a proper sublattice of $L$, because $a = 1$; however, we can still consider the isometry $\phi = \id_T \oplus \psi \in O(T \oplus S)$ defined above. Since $\bar{\psi} = \id_{A_S}$, the isometry $\phi$ can be extended to $\Phi \in O(L)$ by \cite[Corollary 1.5.2]{nikulin}. As explained in \S \ref{subsection: induced isom}, $A_L \cong M^\perp/M$, with $M,M^\perp$ subgroups of $A_T \oplus A_S$, meaning that $\bar{\Phi} = \id_{A_L}$, since $\bar{\phi} = \id \in O(q_{T \oplus S})$. Moreover, we also have $\spinn^{L}(\Phi) = \spinn^{T\oplus S}(\phi) = \spinn^S(\psi) = 1$ (see for instance the proof of \cite[Proposition 3.5]{camere_fourier}). Thus, $\Phi \in \Mo(L)$, and it still has invariant lattice $T$ and co-invariant lattice $S$; we can now apply Theorem \ref{torelli} in the same way as before to conclude that, also in this case, there exists an automorphism of a suitable manifold of $\hskn$-type inducing $\Phi$.  

\emph{Cases (3),(4): $(23,1,0)$ and $(23,1,1)$}. These two cases can be realized by generalizing \cite{bcms_p=23}, where the authors proved the existence of an automorphism of order $23$ on a variety of $\hsk$-type, having invariant lattice $T \cong \langle 46 \rangle$ and co-invariant lattice $S \cong U^{\oplus 2} \oplus E_8^{\oplus 2} \oplus K_{23}$. In Proposition \ref{rkT=1bis} we showed that, if a triple $(p,m,a)$ with $p=23$ is admissible, then $m=1$ and $a \in \left\{ 0,1 \right\}$; moreover, in this case the two orthogonal sublattices of $L$ are $S \cong U^{\oplus 2} \oplus E_8^{\oplus 2} \oplus K_{23}$ and either $T=\langle46(n-1)\rangle$ if $a=1$ (as we have for $n=2$), or $T=\langle \frac{2(n-1)}{23} \rangle$ if $a=0$. We notice in particular that $S$ does not depend on $n$; in \cite[Proposition 5.3]{bcms_p=23} it was proved that such lattice admits an isometry $\psi$ of order $23$ inducing the identity on $A_S$. Thus, $\id_T \oplus \psi \in O(T \oplus S)$ can be extended to an isometry $\phi \in O(L)$ such that $\bar{\phi} = \id \in O(q_L)$ (if $a=0$ we have $L=T \oplus S$, so $\id_T \oplus \psi$ is already an isometry of $L$ with this property; otherwise, if $a=1$, we apply again \cite[Corollary 1.5.2]{nikulin}). Following the same proof of \cite[Theorem 6.1]{bcms_p=23}, there exists an automorphism realizing the triple; we only point out that, while for $n=2$ the monodromies of $L$ are just the isometries preserving the positive cone, for higher values of $n$ the isometry also needs to induce $\pm \id$ on $A_L$ (see \cite[Lemma 9.2]{markman}). This, however, is not a problem since we know that $\bar{\phi} = \id \in O(q_L)$.
\end{proof}

We observed in \S \ref{subsection: rkT=1} that the triple $(3,11,0)$ is admissible for $n=4$, therefore we can now conclude that it is realized by an automorphism: we mark this case with the symbol $\bigstar$ in the corresponding table of Appendix \ref{appendix: tables}. We will see an explicit geometric realization of it in \S \ref{subsection: eightfolds cyclic}.

\section{Induced automorphisms for $n=3,4$, $p=3$} \label{section: induced for n=3,4}

The new admissible triples $(3,m,a)$ that appear passing from $n=2$ to $n=3$ and, more significantly, to $n=4$ (see \S \ref{subsect: classif for n=3,4} and Appendix \ref{appendix: tables}) cannot be realized by natural automorphisms. However, in this section we will show that all of them but one admit a realization using (possibly twisted) induced automorphisms, which were discussed in \S \ref{subsect: induced auto}.

\subsection{Induced automorphisms for $n=4$} \label{subsection: induced autom n=4}

Let $T,S$  be the lattices associated to an admissible triple $(3,m,a)$ for $n=4$ such that $S = S_{K3}$, where $S_{K3}$ is the co-invariant lattice of a non-symplectic automorphism $\varphi$ of order $3$ on a $K3$ surface $\Sigma$ (see \cite{autom_k3_ord3} for a complete classification of these lattices). Let $\hat{T}$ be the orthogonal complement of $S$ in the Mukai lattice $\Lambda_{24}$ (see \S \ref{subsect: induced auto}): since $S = S_{K3}$, we have $\hat{T} \cong T_{K3} \oplus U$, with $T_{K3}$ the invariant lattice of $\varphi$. Then the following result holds.

\begin{lemma} \label{lemma: induced n=4}
If $T_{K3} \cong U(3) \oplus W$, for some even lattice $W$, and $T \cong U \oplus W \oplus \langle -6 \rangle$, then the triple $(3,m,a)$ is realized by an automorphism induced by $\varphi$ on a suitable moduli space $M_{\tau}(v)$. 
\end{lemma}

\begin{proof}
If $T_{K3} \cong U(3) \oplus W$, then $\hat{T} \cong U \oplus U(3) \oplus W$ is the invariant lattice of the extended action of $\varphi$ to $\Lambda_{24}$. Let $v$ be a primitive Mukai vector of square six in the component $U(3)$ of $\hat{T}$: then $v^\perp \cap \hat{T} \cong U \oplus W \oplus \langle -6 \rangle$. Proposition \ref{prop: induced} allows us to conclude.
\end{proof}

\begin{rem}
Lemma \ref{lemma: induced n=4} holds not only for $n=4$, but for any $n$ such that $n \equiv 1$ (mod $3$), since this is the condition which guarantees the existence of an element of square $2(n-1)$ in the lattice $U(3)$.
\end{rem}

\begin{theorem}
For $n=4$, all admissible triples $(3,m,a) \neq (3,11,0), (3,10,3)$, $(3,9,4), (3,8,5)$ admit a geometric realization via non-twisted induced automorphisms.
\end{theorem}

\begin{proof}
Except for the four cases excluded in the statement, the only admissible triples in Table \ref{n=4,ord3} which cannot be realized by a natural automorphism and do not satisfy the hypotheses of Lemma \ref{lemma: induced n=4} are $(3,8,1), (3,7,0), (3,4,1)$ and $(3,3,0)$.

Consider the triple $(3,8,1)$: here we have $S=U^{\oplus 2} \oplus E_6^{\oplus 2}$ and $T = \langle 2 \rangle \oplus E_6$. By \cite[Theorem 3.3]{autom_k3_ord3} there exists a $K3$ surface $\Sigma$ and a non-symplectic automorphism of order three $\varphi \in \aut(\Sigma)$ with $S_{K3} = S$ and $T_{K3} = U \oplus A_2^{\oplus 2}$: in order to show that the triple $(3,8,1)$ is realized by an automorphism induced by $\varphi$ we need to prove the existence of a primitive Mukai vector $v \in \hat{T} = U^{\oplus 2} \oplus A_2^{\oplus 2}$ of square six and orthogonal complement $v^\perp \cap \hat{T}$ isometric to $T$ (see Proposition \ref{prop: induced}). We describe primitive embeddings $\langle 6 \rangle \hookrightarrow \hat{T}$ using Theorem \ref{nik embedd}. The discriminant groups of the two lattices $\langle 6 \rangle$ and $\hat{T}$ are:
\[ A_{\langle 6 \rangle} = \langle s \rangle \cong \frac{\IZ}{6 \IZ}\left( \frac{1}{6}\right); \qquad A_{\hat{T}} = \langle t_1, t_2 \rangle \cong \frac{\IZ}{3 \IZ}\left( \frac{4}{3}\right) \oplus \frac{\IZ}{3 \IZ}\left( \frac{4}{3}\right).\]
We consider the isometric subgroups $H \coloneqq \langle 2s \rangle \subset A_{\langle 6 \rangle}$ and $H' \coloneqq \langle t_1 + t_2 \rangle \subset A_{\hat{T}}$. Let $\gamma: H \ra H'$ be the isomorphism sending the generator of $H$ to the generator of $H'$ (both these elements have order three and quadratic form $\frac{2}{3}$ mod $2\IZ$). The graph of $\gamma$ is the subgroup $\Gamma = \langle 2s + t_1 + t_2 \rangle \subset A_{\langle 6 \rangle} (-1) \oplus  A_{\hat{T}} $ and its orthogonal complement is $\Gamma^\perp = \langle s+t_1, s+ t_2\rangle$. Passing to the quotient $\Gamma^{\perp}/\Gamma$, the class of the element $s + t_2$ becomes the opposite of the class of $s + t_1$, meaning that
\[ \frac{\Gamma^{\perp}}{\Gamma} = \langle [s+t_1] \rangle \cong \frac{\IZ}{6 \IZ} \left(\frac{7}{6} \right).\]

This quotient coincides with the discriminant group of $T = \langle 2 \rangle \oplus E_6$: by Theorem \ref{nik embedd}, this implies that there exists a primitive embedding $\langle 6 \rangle \hookrightarrow \hat{T}$ with orthogonal complement $T$, thus -- by Proposition \ref{prop: induced} -- the triple $(3,8,1)$ has an induced realization. Moreover, this computation guarantees that the triple $(3,4,1)$ is also realized by an induced automorphism, since in this case both $T = \langle 2 \rangle \oplus E_6 \oplus E_8$ and $T_{K3} =  U \oplus A_2^{\oplus 2} \oplus E_8$ differ from the ones of $(3,8,1)$ only for an additional copy of the unimodular lattice $E_8$.

With a similar approach it is possible to show that the admissible triples $(3,7,0)$ and $(3,3,0)$ are realized by induced automorphisms too: here $T = \langle 2 \rangle \oplus E_8$, \mbox{$T_{K3} = U \oplus E_6$} and $T = \langle 2 \rangle \oplus E_8^{\oplus 2}, T_{K3} = U \oplus E_6 \oplus E_8$ respectively. 
\end{proof}

All the cases which can be realized by non-natural, non-twisted induced automorphisms are marked with the symbol $\natural$ in Table \ref{n=4,ord3}.

\subsection{Twisted induced automorphisms for $n=3,4$}
Both for $n=3$ and $n=4$, in \S \ref{subsect: classif for n=3,4} we have found  admissible triples for $p=3$ where the lattice $S$ is different from all possible co-invariant lattices $S_{K3}$ of non-symplectic automorphisms of order three on $K3$ surfaces, classified in \cite{autom_k3_ord3}. Thus, we cannot realize these cases in a natural way, nor using induced automorphisms on moduli spaces of ordinary sheaves on $K3$'s (Proposition \ref{prop: induced}). However, we prove that (excluding $(3,11,0)$ for $n=4$, which will be discussed in \S \ref{subsection: eightfolds cyclic}) they all admit a geometric realization using twisted induced automorphisms (see \S \ref{subsect: induced auto}).

We are interested in the following triples $(p,m,a)$: $(3,9,5)$ and $(3,8,6)$ for \mbox{$n=3$}; $(3,10,3), (3,9,4), (3,8,5)$ for $n=4$. For each of these cases, let $T,S$ be the corresponding lattices in Table \ref{n=3,ord3} ($n=3$) or Table \ref{n=4,ord3} ($n=4$) of Appendix \ref{appendix: tables}. Notice that $S$ is always of the form $S = U(3)^{\oplus 2} \oplus W$, where $W$ is one of the lattices $E_8^{\oplus 2}$, $E_6 \oplus E_8$, $E_6^{\oplus 2}$.

Let $\Sigma$ be a $K3$ surface with transcendental lattice $\trans(\Sigma) = S_{K3} \cong U \oplus U(3) \oplus W$, where $S_{K3}$ is the co-invariant lattice of a non-symplectic automorphism $\varphi \in \aut(\Sigma)$ of order three: the existence of $(\Sigma, \varphi)$ is guaranteed, in all cases, by \cite[Theorem 3.3]{autom_k3_ord3} and \cite[Table 2]{autom_k3_ord3}. This $K3$ surface has $\pic(\Sigma) = T_{K3} \cong U(3) \oplus M$, for an even lattice $M$ which is either $0, A_2, A_2^{\oplus 2}$.

\begin{prop} \label{prop: choice of brauer class}
Let $S$ and $(\Sigma, \varphi)$ be as above. Then there exists a $\varphi$-invariant Brauer class $\alpha \in \br(\Sigma)[3]$ whose kernel in $\trans(\Sigma)$ is isomorphic to $S$.
\end{prop}
 
\begin{proof}
As we recalled in \S \ref{subsect: induced auto}, a Brauer class $\alpha \in \br(\Sigma)$ of order three corresponds to a surjective homomorphism $\alpha: \trans(\Sigma) \ra \IZ/3 \IZ$; it is $\varphi$-invariant if and only if $\alpha \circ \varphi^*\vert_{\trans(\Sigma)} = \alpha$.

We consider $\alpha \coloneqq (e_1, -): \trans(\Sigma) \ra \IZ/3\IZ$, where $e_1$ is one of the generators of the summand $U$ in $\trans(\Sigma) \cong U \oplus U(3) \oplus W$. The kernel of this homomorphism is $\ker(\alpha) = K \oplus U(3) \oplus W$, with $K = \left\{ v \in U : (e_1, v) \equiv 0 \, (\text{mod } 3\IZ) \right\}$: in particular $K = \langle e_1, 3e_2 \rangle \cong U(3)$, thus $\ker(\alpha) \cong S$.

We now want to check that $\alpha \circ \varphi^*\vert_{\trans(\Sigma)} = \alpha$. By \cite[Examples 1.1]{autom_k3_ord3}, recalling that $W$ is a direct sum of copies of $E_6$ and $E_8$, the action of the automorphism $\varphi$ on $\trans(\Sigma) \cong U \oplus U(3) \oplus W$ can be expressed as
\[ \varphi^*\vert_{\trans(\Sigma)} = \pi \oplus \rho\]
\noindent where $\rho$ is a suitable isometry of order three of $W$ with no fixed points and $\pi$ is the isometry of $U \oplus U(3)$ which, with respect to a basis $\left\{ e_1, e_2, f_1, f_2 \right\}$, is given by:
\[ e_1 \mapsto e_1 - f_1, \qquad e_2 \mapsto -2e_2 - f_2,\]
\[f_1 \mapsto -2f_1 + 3e_1, \qquad f_2 \mapsto f_2 + 3 e_2.\]
We have $(e_1, \pi(e_i)) \equiv (e_1,  e_i)$ (mod $3 \IZ$) and $(e_1, \pi(f_i)) \equiv (e_1, f_i) \equiv 0$ (mod $3 \IZ$), for $i=1,2$, therefore the Brauer class $\alpha$ is invariant with respect to $\varphi$.
\end{proof}

\begin{theorem} \label{thm: twisted induced realizations}
The admissible triples $(p,m,a) = (3,9,5), (3,8,6)$ for $n=3$ and $(p,m,a) = (3,10,3), (3,9,4), (3,8,5)$ for $n=4$ admit a geometric realization using twisted induced automorphisms. 
\end{theorem}

\begin{proof}
Fix a triple $(p,m,a)$ as in the statement, let $T,S$ be the invariant and co-invariant lattices associated to it and $\Sigma, \varphi, \alpha$ as in Proposition \ref{prop: choice of brauer class}. We want to construct a moduli space $M_v(\Sigma, \alpha)$ having $T$ as Picard lattice and $S$ as transcendental lattice, and on which $\varphi$ induces an automorphism. 

We are considering $\alpha$ of the form $(e_1, -): \trans(\Sigma) \ra \IZ/3 \IZ$, where $e_1$ is a generator of $U$ inside $\trans(\Sigma) \cong U \oplus U(3) \oplus W$. As a consequence, recalling \S \ref{subsect: induced auto}, the element $B = \frac{e_1}{3} \in \trans(\Sigma) \otimes_{\IZ} \frac{1}{3}\IZ \subset H^2(X, \IQ)$ is a $B$-field lift of $\alpha$, with the properties $B^2 = 0$ and $B \cdot L = 0$ for any $L \in \pic(\Sigma)$.

Assume first that $(p,m,a)$ is one of the three admissible triples for $n=4$. We already remarked that $\pic(\Sigma) = T_{K3} \cong U(3) \oplus M$: this means that we can find a primitive divisor $H$ in the summand $U(3)$ of $\pic(\Sigma)$ with $H^2 = 6$. Moreover, up to taking its opposite we can assume that $H$ is effective (by Riemann--Roch). Let $v = (0,H,0) \in H^*(\Sigma, \IZ)$ be the primitive positive Mukai vector defined by $H$, and $B = \frac{e_1}{3}$ the selected $B$-field lift of $\alpha$: by the properties of $B$ the twisted Mukai vector $v_B$ (defined in \S \ref{subsect: induced auto}) coincides with $v$, and therefore it has square six and it is invariant with respect to $\varphi$.

By the previous discussion, $\varphi$ induces a non-symplectic automorphism of order three on the moduli space of twisted sheaves $M_v(\Sigma, \alpha)$, which is a manifold of $\hskq$-type. The transcendental lattice of $M_v(\Sigma, \alpha)$ is $\ker(\alpha) \cong S$ (Proposition \ref{prop: choice of brauer class}), while its Picard group is isomorphic to the intersection $v_B^\perp \cap \langle \pic(\Sigma), (0,0,1), (3,3B, 0) \rangle$. Since $3B = e_1 \in \trans(\Sigma)$, the lattice generated by $(0,0,1)$ and $(3,3B, 0)$ is orthogonal to $\pic(\Sigma)$; moreover, it is isomorphic to $U(3)$, by the fact that $B^2 = 0$. Thus 
$$\pic\left(M_v(\Sigma, \alpha)\right) \cong \left( H^\perp \cap \pic(\Sigma) \right) \oplus U(3) \cong \langle -6 \rangle \oplus M \oplus U(3)$$
\noindent which is exactly the lattice $T$ corresponding to $(p,m,a)$ (see Table \ref{n=4,ord3}).

Consider now the case where $(p,m,a)$ is one of the admissible triples $(3,9,5)$, $(3,8,6)$ for $n=3$. In this case $\pic(\Sigma) \cong U(3) \oplus A_2 \oplus M'$, therefore -- if $\left\{ e_1, e_2 \right\}$ is a basis for $U(3)$ and $\left\{ \delta_1, \delta_2 \right\}$ a basis for $A_2$ -- we can take the primitive element of square four $\tilde{H} = e_1 + e_2 + \delta_1 \in \pic(\Sigma)$. Let $H$ be the effective divisor between $\tilde{H}$ and $-\tilde{H}$; as before, $v = v_B = (0, H, 0)$ is a primitive positive Mukai vector invariant with respect to $\varphi$. Then $\varphi$ induces an automorphism on $M_v(\Sigma, \alpha)$, which is a manifold of $\hskt$-type with transcendental lattice $\ker(\alpha) \cong S$ and
\[ \pic(M_v(\Sigma, \alpha)) \cong v_B^\perp \cap \langle \pic(\Sigma), (0,0,1), (3,3B, 0) \rangle \cong (H^\perp \cap \pic(\Sigma)) \oplus U(3).\]
It can be shown that the orthogonal complement of $e_1 + e_2 + \delta_1$ in $U(3) \oplus A_2$ is isomorphic to the lattice $\Omega$ defined in \S \ref{subsect: classif for n=3,4}, thus $\pic(M_v(\Sigma, \alpha)) \cong \Omega \oplus M' \oplus U(3)$, which is the lattice $T$ corresponding to the triple $(p,m,a)$ in Table \ref{n=3,ord3}.

To conclude the proof, we need to show that the automorphism induced by $\varphi$ on $M_v(\Sigma, \alpha)$ leaves the whole Picard lattice invariant. Both for $n=3$ and $n=4$, the direct summand $U(3)$ in $\pic\left(M_v(\Sigma, \alpha)\right)$ is the lattice $\langle  (0,0,1), (3,3B, 0) \rangle$: $\varphi$ acts as the identity on $H^4(\Sigma, \IZ)$, therefore $(0,0,1)$ is fixed; moreover, it maps $(3,3B,0)$ to $(3,3 \varphi^*(B),0)$, but these two classes coincide in $H^2(M_v(\Sigma, \alpha), \IZ)$ since they correspond to each other via the Hodge isometry (equivariant with respect to the action of $\varphi)$
$$\exp(\varphi^*(B) - B): \tilde{H}(\Sigma, B, \IZ) \ra \tilde{H}(\Sigma, \varphi^*(B), \IZ)$$
\noindent between the two Hodge structures of $\Sigma$ defined by the $B$-field lifts $B, \varphi^*(B)$ of $\alpha$ (see \cite[\S 2]{stellari_huybrechts} and \cite[Remark 2.4]{ckkm}; here we use $B^2 = \varphi^*(B)^2 = 0$). Since $\varphi^*$ also fixes $\pic(\Sigma)$, we get the result.
\end{proof}

In Table \ref{n=3,ord3} and Table \ref{n=4,ord3} we use the symbol $\diamondsuit$ to mark the five admissible triples which can be realized only via twisted induced automorphisms.

\section{Automorphisms on the LLSvS eightfold} \label{subsection: autom llsvs}

Let $Y \subset \mathbb{P}^5$ be a smooth cubic fourfold which does not contain a plane and $M_3(Y) = \hilb^{\textit{gtc}}(Y)$ the irreducible component of $\hilb^{3n+1}(Y)$ containing twisted cubic curves on $Y$. The manifold $M_3(Y)$ is smooth, projective of dimension ten and it is called the \emph{Hilbert scheme of generalized twisted cubics on $Y$}; in \cite{llsvs}, Lehn, Lehn, Sorger, van Straten proved that there exist an irreducible holomorphic symplectic manifold $Z_Y$ of dimension eight, a closed Lagrangian embedding $j: Y \hookrightarrow Z_Y$ and a morphism $u: M_3(Y) \ra Z_Y$ which factorizes as $\Phi \circ a$, where $a: M_3(Y) \ra Z'_Y$ is a $\mathbb{P}^2$-bundle to an eight-dimensional manifold $Z'_Y$ and $\Phi: Z'_Y \ra Z_Y$ is the contraction of an extremal divisor $D \subset Z'_Y$ to the image $j(Y) \subset Z_Y$. Moreover, by work of Addington and Lehn (\cite{add_lehn}) $Z_Y$ is a manifold of $\hskq$-type.

We recall some details about the construction. For all curves $C \in M_3(Y)$ the linear span $\langle C \rangle \subset \mathbb{P}^5$ is a $\mathbb{P}^3$; in particular, $C$ lies on the cubic surface $S_C = Y \cap \langle C \rangle$, which is integral since $Y$ does not contain any plane. A point $p \in D \subset Z'_Y$ is defined by the datum $(y, \mathbb{P}(W))$, with $y \in Y$ and $\mathbb{P}(W) \subset \mathbb{P}^5$ a three-dimensional linear subspace through $y$ contained in the tangent space $T_y Y$ (here and in the following $W \in \grass(\mathbb{C}^6,4)$). The curves of $M_3(Y)$ parametrized by this datum are non-Cohen--Macaulay: an element $C$ in the fiber $a^{-1}(p)$ is a singular cubic curve $C^0$ cut out on $Y$ by a plane through $y$ contained in $\mathbb{P}(W) \subset T_y Y$, together with an embedded point in $y$. The contraction $\Phi\vert_D: D \ra j(Y)$ sends $p=(y, \mathbb{P}(W))$ to $j(y)$.

Instead, a point $p \in Z'_Y \setminus D$ corresponds to the choice of the following data:
\begin{itemize}
\item a three-dimensional linear subspace $\mathbb{P}(W) \subset \mathbb{P}^5$;
\item a linear determinantal representation for the surface $S = \mathbb{P}(W) \cap Y$, i.e.\ the orbit $[A]$ of a $3\times 3$-matrix $A$ with coefficients in $W^*$ such that $\det(A) = 0$ is an equation for $S$ in $\mathbb{P}(W)$, where the orbit is taken with respect to the action of $\left( \textrm{GL}_3 \times \textrm{GL}_3\right)/\Delta$, $\Delta \coloneqq \left\{ (tI_3, tI_3) : t \in \mathbb{C}\setminus \left\{ 0 \right\} \right\}$ (see \cite[\S 3]{llsvs}).
\end{itemize}
Then, any curve $C$ in the fiber $a^{-1}(p)$ lies on $S$ and is arithmetically-Cohen--Macaulay; the generators of the homogeneous ideal $I_{C/S}$ are the three minors of a $3 \times 2$-matrix $A_0$, whose columns are independent linear combinations of the columns of $A$. The morphism $\Phi$ maps $Z'_Y \setminus D$ isomorphically to $Z_Y \setminus j(Y)$.
 
As observed in \cite[\S 6.2]{bcs}, it is possible to construct non-symplectic automorphisms of the Fano variety of lines $F(Y)$ (which is a manifold of $\hsk$-type, by \cite{beauv_donagi}) starting from automorphisms of the cubic fourfold $Y$. It is therefore natural to ask whether a similar approach can be used to produce automorphisms on $Z_Y$: the answer is positive, we will show how to do so and how to choose $Y$ in order to construct a non-symplectic automorphism on $Z_Y$ realizing the admissible triple $(3,11,0)$ for $n=4$.

By \cite{matsum_monsky}, automorphisms of a cubic hypersurface $Y \subset \mathbb{P}^5$ are restrictions of linear automorphisms of $\mathbb{P}^5$; in particular, the list of all automorphisms of prime order on smooth cubic fourfolds was provided in \cite[Theorem 3.8]{gonzalez_liendo}. 

\begin{lemma}\label{lemma: discesa autom llsvs}
 Let $Y \subset \mathbb{P}^5$ be a smooth cubic fourfold not containing a plane and \mbox{$\sigma \in \mathrm{PGL}(6)$} an automorphism such that $\sigma(Y) = Y$. Then, $\sigma$ induces an automorphism $\check{\sigma}$ of $M_3(Y)$ such that $a(\check{\sigma}(C)) = a(\check{\sigma}(C'))$ if $a(C) = a(C')$.
\end{lemma}

\begin{proof}
We begin by looking at curves in the fibers of $a$ over $D \subset Z'_Y$. Let $p \in D$ be a point corresponding to $(y, \mathbb{P}(W))$, and $C_1, C_2 \in a^{-1}(p)$: as explained above, each $C_i$ consists of a plane cubic curve $C^0_i$, singular in $y$, together with an embedded point at $y$. In particular, $C^0_i = \pi_i \cap Y$, with $\pi_1,\pi_2 $ two-dimensional subspaces inside $\mathbb{P}(W)$ tangent to $Y$ in $y$. Then, $\sigma(C^0_i)$ are again plane cubic curves, cut out on $Y$ by two planes through $\sigma(y)$ inside $\sigma\left( \mathbb{P}(W) \right) \subset T_{\sigma(y)}Y$. Let $\check{\sigma}(C_i)$ be $\sigma(C^0_i)$, with the unique non-reduced structure at $\sigma(y)$: then $\check{\sigma}(C_1)$, $\check{\sigma}(C_2)$ are elements of $M_3(Y)$ in the fiber $a^{-1}(p')$, with $p'$ defined by $(\sigma(y), \sigma\left( \mathbb{P}(W) \right))$.

Consider now a point $p \in Z'_Y \setminus D$, corresponding to $\mathbb{P}(W) \subset \mathbb{P}^5$ and the orbit of a $3 \times 3$-matrix $A= \left( w_{i,j} \right)$, with $w_{i,j} \in W^*$. Denote $\mathbb{P}(W') \coloneqq \sigma\left( \mathbb{P}(W) \right)$ and let $S$ be the integral cubic surface $\mathbb{P}(W) \cap Y$, which is the vanishing locus in $\mathbb{P}(W)$ of $g \coloneqq \det(A) \in S^3 W^*$. Then, the surface $\sigma(S) \subset \mathbb{P}(W')$ is the vanishing locus of $g \circ \sigma^{-1}$, which is the determinant of the matrix $\sigma^* A \coloneqq \left( w_{i,j} \circ \sigma^{-1} \right)$ with coefficients in $\left( W' \right)^*$.

Two elements $C_1, C_2 \in a^{-1}(p)$ are aCM cubic curves on $S$: the generators of $I_{C_i/S}$ are given by the three minors of a $3 \times 2$-matrix $A_i$ whose two columns are in the span of the columns of $A$. Then, $\sigma(C_1), \sigma(C_2)$ are aCM curves on $\sigma(S)$: by pullback, the generators of $I_{\sigma(C_i)/\sigma(S)}$ are the minors of $\sigma^* A_i$, whose columns are again linear combinations of the columns of $\sigma^* A$. Thus, $\check{\sigma}(C_i) \coloneqq \sigma(C_i) \in M_3(Y)$, for $i=1,2$, belongs to the fiber of $a$ over the point defined by $\mathbb{P}(W')$ and $[\sigma^* A]$.  
\end{proof}

As a consequence of Lemma \ref{lemma: discesa autom llsvs}, there exists an automorphism $\sigma'$ of the manifold $Z'_Y$ such that $\sigma' \circ a = a \circ \check{\sigma}$; moreover, from the previous proof, $\sigma'$ leaves the divisor $D$ invariant. Since $\Phi: Z'_Y \ra Z_Y$ is a contraction of $D$, $\sigma'$ descends to an automorphism $\tilde{\sigma} \in \aut(Z_Y)$ (see for instance \cite[Lemma 3.2]{descent}).

By \cite[Proposition 4.8]{voisin}, there exists a dominant rational map of degree six
\[ \psi: F(Y) \times F(Y) \dashrightarrow Z_Y\]
\noindent such that
\begin{equation} \label{mappa voisin}
\psi^* (\omega_{Z_Y}) = \textrm{pr}_1^* (\omega_{F(Y)}) - \textrm{pr}_2^* (\omega_{F(Y)}) 
\end{equation}
\noindent where $\omega_{F(Y)}$ and $\omega_{Z_Y}$ are the symplectic forms on $F(Y)$ and $Z_Y$ respectively. The rational map $\psi$ is defined as follows. Let $(l,l') \in F(Y)\times F(Y)$ be a generic element, so that the span $\langle l,l' \rangle$ is a $\mathbb{P}^3$, and let $x$ be a point on $l$: the plane $\langle x, l' \rangle$ intersects the cubic fourfold $Y$ along the union of the line $l'$ and a conic $Q$ passing through $x$. Then $C \coloneqq l \cup_x Q$ is a rational cubic curve contained in $Y$: we set $\psi(l,l') \coloneqq u(C) \in Z_Y$, which is well-defined since all reducible cubic curves $C$ arising from different choices of the point $x \in l$ belong to the same fiber of $u$.

\begin{lemma}
 Let $Y \subset \mathbb{P}^5$ be a smooth cubic fourfold not containing a plane, \mbox{$\sigma \in \mathrm{PGL}(6)$} such that $\sigma(Y) = Y$ and $\tilde{\sigma} \in \aut(Z_Y)$ the automorphism induced by $\sigma$ on $Z_Y$. Then $\psi\left(\sigma(l), \sigma(l')\right) = \tilde{\sigma}\left( \psi(l,l') \right)$. 
\end{lemma}

\begin{proof}
As we recalled, $\psi(l,l') = u(C)$ with $C = l \cup_x Q$, $x \in l$ and $Y \cap \langle x,l' \rangle = l' \cup Q$; moreover, $\tilde{\sigma}\left( \psi(l,l') \right) = u \left( \check{\sigma}(C) \right)$ by Lemma \ref{lemma: discesa autom llsvs}. In turn, $\psi\left(\sigma(l), \sigma(l')\right) = u(C')$, where $C' = \sigma(l) \cup_{\sigma(x)} Q'$ and $Y \cap \langle \sigma(x), \sigma(l') \rangle = \sigma(l') \cup Q'$. However, the intersection $Y \cap \langle \sigma(x), \sigma(l') \rangle$ coincides with $\sigma \left( Y \cap \langle x, l' \rangle\right)$; as a consequence, $Q' = \sigma(Q)$ and so $C' = \check{\sigma}(C)$.
\end{proof}

Thanks to this equivariance of the map $\psi$ and the relation (\ref{mappa voisin}) we deduce that, if $\tilde{\sigma} \neq \id $ and $\sigma$ acts non-symplectically on $F(Y)$, then $\tilde{\sigma}$ is also non-symplectic and of the same order of $\sigma$.

\begin{prop}\label{prop: transcendental llsvs}
Let $Y \subset \mathbb{P}^5$ be a smooth cubic fourfold not containing a plane. The transcendental lattices of $F(Y)$ and $Z_Y$ have the same rank.
\end{prop}

\begin{proof}
Let $\Gamma_{\psi} \subset F(Y) \times F(Y) \times Z_Y$ be the closure of the graph of the map $\psi: F(Y) \times F(Y) \dashrightarrow Z_Y$ and let $V$ be a desingularization of $\Gamma_\psi$. We consider the projections $\pi_F: V \ra F(Y) \times F(Y)$, $\pi_Z: V \ra Z_Y$ which arise from the inclusion $\Gamma_\psi \subset F(Y) \times F(Y) \times Z_Y$. Let $\trans_{\mathbb{C}}(F(Y)) \subset H^2(F(Y), \mathbb{C})$ and $\trans_{\mathbb{C}}(Z_Y) \subset H^2(Z_Y, \mathbb{C})$ be the complexifications of the transcendental lattices of $F(Y)$ and $Z_Y$ respectively. If we define $\mathcal{T} \coloneqq (\pi_F)_*\left( \pi_Z^* \left( \trans_{\mathbb{C}}(Z_Y) \right)\right)$, using relation (\ref{mappa voisin}) we deduce:
$$\mathcal{T} \subset \trans_{\mathbb{C}}(F(Y)) \oplus \trans_{\mathbb{C}}(F(Y)) \subset H^2(F(Y) \times F(Y), \mathbb{C}).$$
In particular, by the fact that $\psi^*(\omega_{Z_Y}) \in \mathcal{T}$ and the transcendental is the minimal Hodge substructure (in the second cohomology) containing holomorphic two-forms, $(\textrm{pr}_i)_*(\mathcal{T}) = \trans_{\mathbb{C}}(F(Y))$ for $i=1$ or $i=2$. This implies that the ranks of $\trans(Z_Y)$ and $\trans(F(Y))$ coincide. 
\end{proof}

\subsection{The case of cyclic cubic fourfolds} \label{subsection: eightfolds cyclic}
Let $\sigma \in \textrm{PGL}(6)$ be the following automorphism of order three:
\begin{equation} \label{eq: automorphism on cyclic cubics}
\sigma(x_0 : \ldots : x_5) = (x_0 : \ldots : x_4 : \xi x_5)
\end{equation}
\noindent with $\xi = e^{\frac{2 \pi i}{3}}$. We consider the ten-dimensional family $\mathcal{C}$ of smooth hypersurfaces $Y \subset \mathbb{P}^5$ of equations
\begin{equation*}
Y: x_5^3 + F_3(x_0, \ldots, x_4) = 0
\end{equation*} 
\noindent with $F_3$ an homogeneous polynomial of degree three. Cubic fourfolds $Y \in \mathcal{C}$ are called \emph{cyclic}: they arise as triple coverings of $\mathbb{P}^4$ ramified along the smooth cubic threefold of equation $F_3 = 0$. Any $Y \in \mathcal{C}$ is invariant with respect to $\sigma$, thus $\sigma\vert_Y \in \aut(Y)$.

\begin{rem}\label{rem: cubic fourfold no planes}
In \cite[Example 6.4]{bcs} it was proved that $\sigma$ induces a non-symplectic automorphism of order three on the Fano variety of lines $F(Y)$, whose invariant lattice is $T' = \langle 6 \rangle$. This allows us to deduce that a very general $Y$ in the family $\mathcal{C}$ does not contain any plane: in fact, if there existed a plane $\pi \subset Y$, it would define an algebraic class in $H^{2,2}(Y)$; in particular, the second N\'eron--Severi group $\ns_2(Y) = H^4(Y, \IZ) \cap H^{2,2}(Y)$ would contain $\langle H^2, \pi\rangle$, where $H$ is an ample line bundle on $Y$, thus $\rk(\ns_2(Y)) \geq 2$ (references in \cite{macrì_stellari}). By applying the Abel--Jacobi map $H^{2,2}(Y) \ra H^{1,1}(F(Y))$ (see \cite{beauv_donagi}), the Picard group of $F(Y)$ would also have at least rank $2$, while we know that $\pic(F(Y)) \cong T'$, for a very general choice of $Y$. 
\end{rem}

We can then construct the manifold $Z_Y$, for $Y$ very general in the family $\mathcal{C}$, and consider $\tilde{\sigma} \in \aut(Z_Y)$: by our remarks at the end of the previous subsection, it is a non-symplectic automorphism of order three. We obtain the following result as a corollary of Proposition \ref{prop: transcendental llsvs}.

\begin{restatable}{cor}{corollaryInvLLSVS} \label{cor: invariant on llsvs}
Let $Y$ be a cubic fourfold in the family $\mathcal{C}$ not containing a plane and $\tilde{\sigma} \in \aut(Z_Y)$ the automorphism induced by $\sigma \in \aut(Y)$ of the form (\ref{eq: automorphism on cyclic cubics}). Then, the invariant lattice of $\tilde{\sigma}$ is $T \cong \langle 2 \rangle$.
\end{restatable}

\begin{proof}
As explained in Remark \ref{rem: cubic fourfold no planes}, the very general cubic fourfold $Y \in \mathcal{C}$ is such that $F(Y)$ has transcendental lattice of rank $22$. This, together with Proposition \ref{prop: transcendental llsvs}, allows us to conclude that the invariant lattice of $\tilde{\sigma}$ has the same rank of the invariant lattice of the automorphism induced by $\sigma$ on $F(Y)$, namely one; therefore, by Proposition \ref{rkT=1bis}, $T \cong \langle 2 \rangle$ .   
\end{proof}

At the end of this section we will present a more geometric proof of Corollary \ref{cor: invariant on llsvs}, using Theorem \ref{thm: llsvs maximal family}. In order to do so, we first need to study the fixed locus of the automorphism $\tilde{\sigma}$.

Let $H \subset \textrm{Fix}(\sigma)$ be the hyperplane $\left\{x_5 = 0 \right\} \subset \mathbb{P}^5$; the intersection $Y_H \coloneqq Y \cap H$ is the smooth cubic threefold defined by $F_3(x_0, \ldots, x_4) = 0$ inside $H$. We denote by $Z_H$ the image via the map $u: M_3(Y) \ra Z_Y$ of the set of twisted cubics contained in $Y_H$: in \cite[Proposition 2.9]{shinder_soldatenkov} the authors prove that $Z_H$ is a Lagrangian subvariety of $Z$.

\begin{lemma}\label{lemma: Z_H in fixed locus}
Let $Y$ be a cubic fourfold in the family $\mathcal{C}$ not containing a plane. Then $Z_H$ is contained in the fixed locus of $\tilde{\sigma}$ and any fixed point in $j(Y)$ belongs to $Z_H$.
\end{lemma}
\begin{proof}
Let $j(y)$ be a point in the image of the embedding $j: Y \hookrightarrow Z_Y$ such that $\tilde{\sigma}(j(y)) = j(y)$. In the proof of Lemma \ref{lemma: discesa autom llsvs} we showed that $\check{\sigma} \in \aut(M_3(Y))$ maps the fiber of $u: M_3(Y) \ra Z_Y$ over the point $j(y)$ to the fiber over $j(\sigma(y))$. Therefore, since $\tilde{\sigma} \circ u = u \circ \check{\sigma}$, we need $\sigma(y) = y$, i.e.\ $y \in Y_H$. We conclude $\textrm{Fix}(\tilde{\sigma}) \cap j(Y) = j(Y_H)$. Clearly, since $H \subset \textrm{Fix}(\sigma)$, we have $Z_H \subset \textrm{Fix}(\tilde{\sigma})$; moreover, $Z_H \cap j(Y) \cong Y_H$ (see \cite[\S 3]{shinder_soldatenkov}), thus $Z_H \cap j(Y) = j(Y_H)$.
\end{proof}

\begin{prop}\label{prop: fixed locus on llsvs}
For $Y$ in the family $\mathcal{C}$ not containing a plane, the fixed locus of the automorphism $\tilde{\sigma}$ is $Z_H$. 
\end{prop}

\begin{proof}
By Lemma \ref{lemma: Z_H in fixed locus}, we need to prove that there are no fixed points outside $Z_H$, i.e.\ points $p \in Z_Y \setminus j(Y)$ fixed by $\tilde{\sigma}$ such that the curves in the fiber $u^{-1}(p)$ are not contained in $H$. Notice that a point $p$ of this type corresponds to $(\mathbb{P}(W), [A])$, with $\sigma(\mathbb{P}(W)) = \mathbb{P}(W)$ but $\sigma\vert_{\mathbb{P}(W)} \neq \id$. A vector space $W \in \grass(\mathbb{C}^6,4)$ is $\sigma$-invariant if and only if it can be written as $W = W_1 \oplus W_\xi$, where we set \mbox{$W_t \coloneqq \left\{ w \in W \mid \sigma(w) = tw \right\}$}. The condition $\sigma\vert_{\mathbb{P}(W)} \neq \id$ implies $W_\xi \neq 0$, therefore $W_\xi$ is the whole one-dimensional eigenspace of $\mathbb{C}^6$ with respect to the eigenvalue $\xi$ of $\sigma$, while $W_1$ is a three-dimensional subspace of the eigenspace of $\mathbb{C}^6$ where $\sigma$ acts as the identity. Let $y_0,y_1,y_2 \in W^*$ be the dual elements of a basis of $W_1$; then, we can take $y_0,y_1,y_2,x_5$ as coordinates on $\mathbb{P}(W)$, so that the action of $\sigma$ on it is $\sigma(y_0:y_1:y_2:x_5) = (y_0: y_1:y_2:\xi x_5)$.

We showed in the proof of Lemma \ref{lemma: discesa autom llsvs} that, for a point $p$ as above, we have $\tilde{\sigma}(p) = (\mathbb{P}(W), [\sigma^* A])$. Therefore, $p$ is fixed if and only if the matrices $A$ and $\sigma^* A$ define the same $\mathbb{P}^2$ of generalized aCM twisted cubics on the surface $S=\mathbb{P}(W) \cap Y$, whose equation in $\mathbb{P}(W)$ is of the form $g \coloneqq x_5^3 + f(y_0,y_1,y_2) = 0$, where $f$ is the restriction of $F_3$ to $\mathbb{P}(W_1)$. 
Fix a curve $C$ in the fiber $u^{-1}(p)$: its equations in $\mathbb{P}(W)$ are given by the three minors of a $3\times 2$-matrix $A_0$ with linear entries in $W^*$. In particular, the matrix $A_0$ -- up to a change of basis -- can only be of eight different types, listed in \cite[\S 1]{llsvs}. Since the curve $C$ lies on $S$, the polynomial $g$ defining the surface belongs to $I_{C/S}$, i.e.\ it is a combination of the minors of $A_0$ (see \cite[\S 3.1]{llsvs}). We recall that $A$ is a linear determinantal representation of the surface $S$ therefore, without loss of generality, it is of the form
\begin{equation*}
\renewcommand{\arraystretch}{0.5}
A = \left( \begin{array}{c|c}
A_0 & \begin{matrix}
* \\ * \\ *
\end{matrix}
\end{array} \right)
\end{equation*}
\noindent where the last column is uniquely determined by $g$ (up to a combination of the columns of $A_0$). By \cite[\S 3.1]{llsvs}, the matrices $A$, $\sigma^* A$ define the same $\mathbb{P}^2$ of cubics on $S$ if and only if the columns of $\sigma^* A_0$ belong to the span of the columns of $A$. 

Assume $A_0$ is of the most general form, i.e.\ the form $A^{(1)} = \begin{pmatrix}w_0 & w_1 & w_2 \\ w_1 & w_2 & w_3
\end{pmatrix}^t$ of \cite[\S 1]{llsvs}, where $w_0, \dots, w_3$ are suitable coordinates for $\mathbb{P}(W)$: in this case $C$ is a smooth twisted cubic curve.

Let $M \coloneqq \left(a_{i,j} \mid b_i \right)_{\substack{i=0,1,2, 3 \\ j=0,1,2}} \in \textrm{GL}_4(\mathbb{C})$ be the matrix defining the change of coordinates from $\left\{ w_i \right\}_{i=0}^3$ to $\left\{ y_0, y_1, y_2, x_5 \right\}$; then
\[
\renewcommand{\arraystretch}{1.5}
A = \left( \begin{array}{cc|c}
\sum_{j=0}^2 a_{0,j}y_j + b_0 x_5 & \sum_{j=0}^2 a_{1,j}y_j + b_1 x_5 & \sum_{j=0}^2 c_{0,j}y_j + d_0 x_5 \\
\sum_{j=0}^2 a_{1,j}y_j + b_1 x_5 & \sum_{j=0}^2 a_{2,j}y_j + b_2 x_5 & \sum_{j=0}^2 c_{1,j}y_j + d_1 x_5 \\
\sum_{j=0}^2 a_{2,j}y_j + b_2 x_5 & \sum_{j=0}^2 a_{3,j}y_j + b_3 x_5 & \sum_{j=0}^2 c_{2,j}y_j + d_2 x_5
\end{array} \right) \]
\noindent where the parameters $c_{i,j}$ and $d_i$ are determined by $g$. Once we apply the automorphism we get:
\[
\renewcommand{\arraystretch}{1.5}
\sigma^* A = \left( \begin{array}{cc|c}
\sum_{j=0}^2 a_{0,j}y_j + \xi^2 b_0 x_5 & \sum_{j=0}^2 a_{1,j}y_j + \xi^2 b_1 x_5 & \sum_{j=0}^2 c_{0,j}y_j + \xi^2 d_0  x_5 \\
\sum_{j=0}^2 a_{1,j}y_j + \xi^2 b_1 x_5 & \sum_{j=0}^2 a_{2,j}y_j +\xi^2  b_2 x_5 & \sum_{j=0}^2 c_{1,j}y_j + \xi^2 d_1 x_5 \\
\sum_{j=0}^2 a_{2,j}y_j + \xi^2 b_2 x_5 & \sum_{j=0}^2 a_{3,j}y_j + \xi^2 b_3 x_5 & \sum_{j=0}^2 c_{2,j}y_j + \xi^2 d_2 x_5
\end{array} \right) \]
\noindent where the first two columns form the matrix $\sigma^* A_0$, whose minors define the curve $\sigma(C) \subset S$.

The first column of $\sigma^* A_0$ is a $\mathbb{C}$-linear combination of the columns of $A$ if and only if the following linear system of twelve equations admits a solution $(h,k,t) \in \mathbb{C}^3$:
\begin{equation}\label{sistema equazioni llsvs}
\begin{cases}
a_{0,j} = h a_{0,j} + k a_{1,j} + t c_{0,j} \qquad \text{for } j=0,1,2\\
a_{1,j} = h a_{1,j} + k a_{2,j} + t c_{1,j} \qquad \text{for } j=0,1,2\\
a_{2,j} = h a_{2,j} + k a_{3,j} + t c_{2,j} \qquad \text{for } j=0,1,2\\
\xi^2 b_0 = h b_0 + k b_1 + t d_0 \\
\xi^2 b_1 = h b_1 + k b_2 + t d_1 \\
\xi^2 b_2 = h b_2 + k b_3 + t d_2 
\end{cases}
\end{equation}

Notice that the system made of the last three equations always admits a unique solution, namely $(h,k,t) = (\xi^2, 0,0)$. In fact, the determinant of its matrix of coefficients is different from zero, because it coincides with the coefficient of $x_5^3$ in the expression of the determinant of $A$ (and $\sigma^* A$), which needs to be a (non-zero) scalar multiple of $g$. Now, the triple $(h,k,t) = (\xi^2, 0,0)$ is a solution for the whole system (\ref{sistema equazioni llsvs}) only if $a_{0,j} = a_{1,j} = a_{2,j} = 0 \; \forall j=0,1,2$, which is not possible since the matrix $M$ needs to be invertible. We conclude that the columns of $\sigma^* A_0$ can never be combinations of the columns of $A$ if $A_0$ is of the form $A^{(1)}$. The remaining cases, i.e.\ $A_0$ of the forms $A^{(2)},\dots, A^{(8)}$  of \cite[\S 1]{llsvs}, can be discussed in an entirely similar way.
\end{proof}

Let us fix a cubic fourfold $Y \in \mathcal{C}$ not containing a plane and choose a marking $\eta_0^{}: H^2(Z_Y, \IZ) \ra L = U^{\oplus 3} \oplus E_8^{\oplus 2} \oplus \langle -6 \rangle$. We define $\rho \coloneqq \eta_0^{} \circ \left( \tilde{\sigma} \right)^* \circ \eta_0^{-1} \in O(L)$. Following \cite{bcs_ballq} and \cite{bcs_cubic}, a \emph{$(\rho, \langle 2 \rangle)$-polarization} of an IHS manifold $X$ of $\hskq$-type is given by a marking $\eta: H^2(X, \IZ) \ra L$ and an automorphism $g \in \aut(X)$ of order three such that $g\vert_{H^{2,0}(X)} = \xi \id$ and $\eta \circ g^* = \rho \circ \eta$ (in particular, the invariant lattice of $g$ is isometric to $\langle 2 \rangle$, by Corollary \ref{cor: invariant on llsvs}). We consider the following equivalence relation: two $(\rho, \langle 2 \rangle)$-polarized eightfolds $(X, \eta, g)$, $(X', \eta', g')$ are equivalent if there exists an isomorphism $f: X \ra X'$ such that $\eta' = \eta \circ f^*$ and $g' = f \circ g \circ f^{-1}$. Let $\mathcal{M}^{\rho, \xi}_{\langle 2 \rangle}$ be the set of equivalence classes of $(\rho, \langle 2 \rangle)$-polarized manifolds of $\hskq$-type and $\mathcal{U} \subset \mathcal{M}^{\rho, \xi}_{\langle 2 \rangle}$ the subset which parametrizes manifolds $(Z_Y, \eta, \tilde{\sigma})$, with $Y$ cyclic cubic fourfold not containing a plane and $\sigma$ as in (\ref{eq: automorphism on cyclic cubics}).

For any smooth cubic threefold $\mathcal{J} \subset \mathbb{P}^4$, we denote by $Y(\mathcal{J})$ the cubic fourfold which arises as triple covering of $\mathbb{P}^4$ ramified along $\mathcal{J}$. Using Proposition \ref{prop: fixed locus on llsvs} we can prove the following result.

\begin{theorem} \label{thm: llsvs maximal family}
Let $\mathcal{J},\mathcal{J}'$ be smooth cubic threefolds such that $Y(\mathcal{J}), Y(\mathcal{J'})$ do not contain a plane. If $(Z_{Y(\mathcal{J})}, \eta, \tilde{\sigma})$, $(Z_{Y(\mathcal{J}')}, \eta', \tilde{\sigma}')$ are equivalent as $(\rho, \langle 2 \rangle)$-polarized manifolds, then $\mathcal{J} \cong \mathcal{J}'$. In particular, $\mathcal{U} \subset \mathcal{M}^{\rho, \xi}_{\langle 2 \rangle}$ has dimension ten.
\end{theorem}

\begin{proof}
Consider $(Z_Y, \eta, \tilde{\sigma}) \in \mathcal{U}$ and let $Z_H \subset Z_Y$ be the fixed locus of $\tilde{\sigma}$: by \cite[Theorem 3.3]{shinder_soldatenkov} and \cite[\S 6.3]{iliev}, $Z_H$ also arises as resolution of the unique singular point of the theta divisor in the intermediate Jacobian $\textrm{J}(Y_H)$ of the cubic threefold $Y_H$. This implies that $Z_H$ is a variety of maximal Albanese dimension and \mbox{$\alb(Z_H) \cong \textrm{J}(Y_H)$} (see for instance \cite[\S 1]{maximal_albanese} and references therein). By the Torelli theorem for cubic threefolds (\cite[Theorem 13.11]{clemens_griffiths}), we conclude that the eightfold $Z_Y$ and the action of the automorphism $\tilde{\sigma}$ uniquely determine the threefold $\mathcal{J} = Y_H$ up to isomorphisms. The moduli space $\mathcal{C}^{\textit{sm}}_3$ of smooth cubic threefolds is ten-dimensional and, for $\mathcal{J} \in \mathcal{C}^{\textit{sm}}_3$ very general, the cubic fourfold $Y(\mathcal{J})$ does not contain a plane (see Remark \ref{rem: cubic fourfold no planes}). Since $\mathcal{M}^{\rho, \xi}_{\langle 2 \rangle}$ is ten-dimensional too by \cite[Corollary 6.5]{bcs_ballq}, ten is also the dimension of the subset $\mathcal{U} \subset \mathcal{M}^{\rho, \xi}_{\langle 2 \rangle}$.
\end{proof}

Theorem \ref{thm: llsvs maximal family} allows us to provide the following alternative proof of Corollary \ref{cor: invariant on llsvs}. The automorphism $\tilde{\sigma}$ corresponds to an admissible triple $(3, m ,a)$, where $m-1$ coincides with the dimension of the moduli space $\mathcal{U}$ by \cite[\S 4]{bcs}. Since the dimension of $\mathcal{U}$ is ten, we can  use Proposition \ref{rkT=1bis} to deduce $m = 11, a = 0$; hence the invariant lattice of $\tilde{\sigma}$ is $T \cong \langle 2 \rangle$. 

\clearpage

\appendix
\section{Invariant and co-invariant lattices for $n=3,4$, $p=3$} \label{appendix: tables}

\begin{table}[h]
 \begin{tabular}{c|c|c|c|c}
 $p$&$m$&$a$&$S$&$T$\\
 \hline
  $\clubsuit$ 3&10&0&$U^{\oplus 2}\oplus E_8^{\oplus 2}$&$U\oplus \langle -4\rangle$\\
 
 $\clubsuit$ 3&10&2&$U\oplus U(3)\oplus E_8^{\oplus 2}$&$U(3)\oplus\langle -4\rangle$\\
 \hline
 
$\clubsuit$ 3&9&1&$U^{\oplus 2}\oplus E_6\oplus E_8$&$U\oplus A_2\oplus \langle -4\rangle$\\
 
 $\clubsuit$ 3&9&3&$U\oplus U(3)\oplus E_6\oplus E_8$&$U(3)\oplus A_2\oplus \langle -4\rangle$\\
 
 $\diamondsuit$ 3&9&5&$U(3)^{\oplus 2}\oplus E_6\oplus E_8$& $U(3) \oplus \Omega$\\

 \hline
 
 $\clubsuit$ 3&8&2&$U^{\oplus 2}\oplus E_6^{\oplus 2}$&$U\oplus A_2^{\oplus 2}\oplus \langle -4\rangle$\\
 
 $\clubsuit$ 3&8&4&$U\oplus U(3)\oplus E_6^{\oplus 2}$&$U(3)\oplus A_2^{\oplus 2}\oplus \langle -4\rangle$\\
 
 $\diamondsuit$ 3&8&6&$U(3)^{\oplus 2}\oplus E_6^{\oplus 2}$& $U(3) \oplus A_2 \oplus \Omega$\\

 \hline
 
 $\clubsuit$ 3&7&1&$ U^{\oplus 2} \oplus A_2 \oplus E_8 $&$U\oplus E_6\oplus \langle -4\rangle$\\
 
 $\clubsuit$ 3&7&3& $U\oplus U(3)\oplus A_2\oplus E_8$&$U\oplus A_2^{\oplus 3}\oplus \langle -4\rangle$\\
 
 $\clubsuit$ 3&7&5&$U^{\oplus 2}\oplus A_2^{\oplus 5}$&$U(3)\oplus A_2^{\oplus 3}\oplus \langle -4\rangle$\\
 
 $\clubsuit$ 3&7&7&$U\oplus U(3)\oplus A_2^{\oplus 5}$&$U(3)\oplus E_6^\vee(3)\oplus \langle -4\rangle$\\
 \hline
 $\clubsuit$ 3&6&0&$U^{\oplus 2}\oplus E_8$&$U\oplus E_8\oplus \langle -4\rangle$\\
 
 $\clubsuit$ 3&6&2&$U\oplus U(3)\oplus E_8$&$U\oplus E_6\oplus A_2\oplus \langle -4\rangle$\\
 
 $\clubsuit$ 3&6&4&$U^{\oplus 2}\oplus A_2^{\oplus 4}$&$U\oplus A_2^{\oplus 4}\oplus \langle -4\rangle$\\
 
 $\clubsuit$ 3&6&6&$U\oplus U(3)\oplus A_2^{\oplus 4}$&$U(3)\oplus A_2^{\oplus 4}\oplus \langle -4\rangle$\\
 \hline
 
 $\clubsuit$ 3&5&1&$U^{\oplus 2}\oplus E_6$&$U\oplus E_8\oplus A_2\oplus \langle -4\rangle$\\
 
 $\clubsuit$ 3&5&3&$U\oplus U(3)\oplus E_6$&$U\oplus A_2^{\oplus 2}\oplus E_6\oplus \langle -4\rangle$\\
 
 $\clubsuit$ 3&5&5&$U\oplus U(3)\oplus A_2^{\oplus 3}$&$U\oplus A_2^{\oplus 5}\oplus \langle -4\rangle$\\
 \hline

 $\clubsuit$ 3&4&2&$U^{\oplus 2}\oplus A_2^{\oplus 2}$&$U\oplus E_6^{\oplus 2} \oplus \langle -4\rangle$\\
 
 $\clubsuit$ 3&4&4&$U\oplus U(3)\oplus A_2^{\oplus 2}$&$U\oplus E_6\oplus A_2^{\oplus 3}\oplus \langle -4\rangle$\\

 \hline
 
 $\clubsuit$ 3&3&1&$U^{\oplus 2}\oplus A_2$&$U\oplus E_6\oplus E_8\oplus \langle -4\rangle$\\
 
 $\clubsuit$ 3&3&3&$U\oplus U(3)\oplus A_2$&$U\oplus E_6^{\oplus 2}\oplus A_2 \oplus \langle -4\rangle$\\
 \hline
 $\clubsuit$ 3&2&0&$U^{\oplus 2}$&$U\oplus E_8^{\oplus 2}\oplus \langle -4\rangle$\\
 
 $\clubsuit$ 3&2&2&$U\oplus U(3)$&$U\oplus E_6\oplus E_8\oplus A_2\oplus \langle -4\rangle$\\
 \hline
 
 $\clubsuit$ 3&1&1&$A_2(-1)$&$U\oplus E_8^{\oplus 2}\oplus A_2\oplus \langle -4\rangle$\\
 
 \end{tabular}
 \vspace*{2mm}
  \caption{$n=3$, $p=3$. See \S \ref{subsect: classif for n=3,4} for the definition of the lattice $\Omega$.}\label{n=3,ord3}
 \end{table}
 
 \begin{table}[p]
 \begin{tabular}{c|c|c|c|c}
 $p$&$m$&$a$&$S$&$T$\\
 \hline
 
 $\bigstar$ 3&11&0&$U^{\oplus 2}\oplus E_8^{\oplus 2}\oplus A_2$&$\langle 2\rangle$\\
 \hline
 
 $\clubsuit$ 3&10&0& $U^{\oplus 2} \oplus E_8^{\oplus 2}$&$U\oplus \langle -6\rangle$\\
 
 $\natural$ 3&10&1& $U \oplus U(3) \oplus E_8^{\oplus 2}$&$U\oplus \langle -6\rangle$\\
 
 $\clubsuit$ 3&10&2&$U \oplus U(3) \oplus E_8^{\oplus 2}$&$U(3)\oplus\langle -6\rangle$\\
 
 $\diamondsuit$ 3&10&3&$U(3)^{\oplus 2} \oplus E_8^{\oplus 2}$&$U(3)\oplus\langle -6\rangle$\\
 \hline
 
 $\clubsuit$ 3&9&1&$U^{\oplus 2} \oplus E_6\oplus E_8$&$U\oplus A_2\oplus \langle -6\rangle$\\
 
 $\natural$ 3&9&2&$U\oplus U(3) \oplus E_6\oplus E_8$&$U\oplus A_2\oplus \langle -6\rangle$\\
 
 $\clubsuit$ 3&9&3&$U\oplus U(3)\oplus E_6\oplus E_8$&$U(3)\oplus A_2\oplus \langle -6\rangle$\\
 
 $\diamondsuit$ 3&9&4&$U(3)^{\oplus 2}\oplus E_6\oplus E_8$&$U(3)\oplus A_2\oplus \langle -6\rangle$\\

 \hline
 
 $\natural$ 3&8&1&$U^{\oplus 2} \oplus E_6^{\oplus 2}$&$\langle 2\rangle \oplus E_6$\\
 
 $\clubsuit$ 3&8&2&$U^{\oplus 2} \oplus E_6^{\oplus 2}$&$U\oplus A_2^{\oplus 2} \oplus \langle -6\rangle$\\
 
 $\natural$ 3&8&3&$U\oplus U(3)\oplus E_6^{\oplus 2}$&$U\oplus A_2^{\oplus 2} \oplus \langle -6\rangle$\\
 
 $\clubsuit$ 3&8&4&$U\oplus U(3)\oplus E_6^{\oplus 2}$&$U(3)\oplus A_2^{\oplus 2} \oplus \langle -6\rangle$\\
 
 $\diamondsuit$ 3&8&5&$U(3)^{\oplus 2}\oplus E_6^{\oplus 2}$&$U(3)\oplus A_2^{\oplus 2} \oplus \langle -6\rangle$\\
 
 \hline
 
 $\natural$ 3&7&0&$ U^{\oplus 2} \oplus A_2\oplus E_8 $&$\langle 2\rangle \oplus E_8$\\ 
 
 $\clubsuit$ 3&7&1&$ U^{\oplus 2} \oplus A_2\oplus E_8$&$U\oplus E_6\oplus \langle -6\rangle$\\
 
 $\natural$ 3&7&2&$U\oplus U(3)\oplus A_2\oplus E_8$&$U\oplus E_6\oplus \langle -6\rangle$\\
 
 $\clubsuit$ 3&7&3& $U\oplus U(3)\oplus A_2\oplus E_8$&$U(3) \oplus E_6 \oplus \langle -6\rangle$\\
 
 $\natural$ 3&7&4&$U^{\oplus 2} \oplus A_2^{\oplus 5}$&$U \oplus A_2^{\oplus 3}\oplus \langle -6\rangle$\\
  
 $\clubsuit$ 3&7&5&$U^{\oplus 2} \oplus A_2^{\oplus 5}$&$U(3)\oplus A_2^{\oplus 3}\oplus \langle -6\rangle$\\
 
 $\natural$ 3&7&6&$U\oplus U(3)\oplus A_2^{\oplus 5}$&$U\oplus E_6^\vee(3)\oplus \langle -6\rangle$\\
 
 $\clubsuit$ 3&7&7&$U\oplus U(3)\oplus A_2^{\oplus 5}$&$U(3)\oplus E_6^\vee(3)\oplus \langle -6\rangle$\\
 \hline
 $\clubsuit$ 3&6&0&$U^{\oplus 2} \oplus E_8$&$U\oplus E_8\oplus \langle -6\rangle$\\
 
 $\natural$ 3&6&1&$U\oplus U(3)\oplus E_8$&$U\oplus E_8\oplus \langle -6\rangle$\\
 
 $\clubsuit$ 3&6&2&$U\oplus U(3)\oplus E_8$&$U(3)\oplus E_8\oplus \langle -6\rangle$\\
 
 $\natural$ 3&6&3&$U^{\oplus 2} \oplus A_2^{\oplus 4}$&$U\oplus E_6\oplus A_2\oplus \langle -6\rangle$\\
 
 $\clubsuit$ 3&6&4&$U^{\oplus 2} \oplus A_2^{\oplus 4}$&$U(3)\oplus E_6 \oplus A_2\oplus \langle -6\rangle$\\
 
 $\natural$ 3&6&5&$U\oplus U(3)\oplus  A_2^{\oplus 4}$&$U\oplus A_2^{\oplus 4}\oplus \langle -6\rangle$\\
 
 $\clubsuit$ 3&6&6&$U\oplus U(3)\oplus  A_2^{\oplus 4}$&$U(3)\oplus A_2^{\oplus 4}\oplus \langle -6\rangle$\\
 \hline
 
 $\clubsuit$ 3&5&1&$U^{\oplus 2}\oplus E_6$&$U\oplus E_8\oplus A_2\oplus \langle -6\rangle$\\
 
 $\natural$ 3&5&2&$U\oplus U(3)\oplus E_6$&$U\oplus E_8\oplus A_2\oplus \langle -6\rangle$\\
 
 $\clubsuit$ 3&5&3&$U\oplus U(3)\oplus E_6$&$U(3) \oplus E_8 \oplus A_2\oplus \langle -6\rangle$\\
 
 $\natural$ 3&5&4&$U\oplus U(3)\oplus A_2^{\oplus 3}$&$U\oplus E_6 \oplus A_2^{\oplus 2}\oplus \langle -6\rangle$\\
 
 $\clubsuit$ 3&5&5&$U\oplus U(3)\oplus A_2^{\oplus 3}$&$U(3)\oplus E_6 \oplus A_2^{\oplus 2}\oplus \langle -6\rangle$\\
 \hline 

 $\natural$ 3&4&1&$U^{\oplus 2} \oplus A_2^{\oplus 2}$&$\langle 2\rangle \oplus E_6 \oplus E_8$\\ 
 
 $\clubsuit$ 3&4&2&$U^{\oplus 2} \oplus A_2^{\oplus 2}$&$U\oplus A_2^{\oplus 2}\oplus E_8 \oplus \langle -6\rangle$\\
 
 $\natural$ 3&4&3&$U\oplus U(3)\oplus A_2^{\oplus 2}$&$U\oplus E_6^{\oplus 2}\oplus \langle -6\rangle$\\
 
 $\clubsuit$ 3&4&4&$U\oplus U(3)\oplus A_2^{\oplus 2}$&$U(3)\oplus E_6^{\oplus 2}\oplus \langle -6\rangle$\\
 \hline
 
 $\natural$ 3&3&0&$U^{\oplus 2}\oplus A_2$&$\langle 2 \rangle \oplus E_8^{\oplus 2}$\\
 
 $\clubsuit$ 3&3&1&$U^{\oplus 2}\oplus A_2$&$U\oplus E_6\oplus E_8\oplus \langle -6\rangle$\\
 
 $\natural$ 3&3&2&$U\oplus U(3)\oplus A_2$&$U\oplus E_6\oplus E_8\oplus \langle -6\rangle$\\
 
 $\clubsuit$ 3&3&3&$U\oplus U(3)\oplus A_2$&$U(3)\oplus E_6\oplus E_8 \oplus \langle -6\rangle$\\
 \hline
 $\clubsuit$ 3&2&0&$U^{\oplus 2}$&$U\oplus E_8^{\oplus 2}\oplus \langle -6\rangle$\\
 
 $\natural$ 3&2&1&$U\oplus U(3)$&$U\oplus E_8^{\oplus 2}\oplus \langle -6\rangle$\\
 
 $\clubsuit$ 3&2&2&$U\oplus U(3)$&$U(3)\oplus E_8^{\oplus 2}\oplus \langle -6\rangle$\\
 \hline
 
 $\clubsuit$ 3&1&1&$A_2(-1)$&$U\oplus E_8^{\oplus 2}\oplus A_2\oplus \langle -6\rangle$\\
 \end{tabular}
 \vspace*{2mm}
 \caption{$n = 4$, $p = 3$}\label{n=4,ord3}
 \end{table}

\clearpage
\bibliographystyle{amsplain}
\bibliography{NonSymplecticBiblio}
\end{document}